\documentclass[11pt,oneside,english,jou]{amsart}
\usepackage[utf8]{inputenc} 
\usepackage[T1]{fontenc}
\usepackage{babel}
\usepackage{verbatim}
\usepackage{mathrsfs}
\usepackage{amstext}
\usepackage{amsthm}
\usepackage{amssymb}
\usepackage{amsfonts}
\usepackage{amsmath}
\usepackage{esint}
\usepackage{enumerate}
\usepackage[unicode=true,pdfusetitle,
bookmarks=true,bookmarksnumbered=false,bookmarksopen=false,
breaklinks=false,pdfborder={0 0 1},colorlinks=false]
{hyperref}
\usepackage[margin=2.5cm]{geometry}
\usepackage[foot]{amsaddr}
\usepackage{mathtools}
\usepackage{ dsfont }
\usepackage[shortlabels]{enumitem}
\usepackage{bm}
\usepackage{caption}

\usepackage[markup=underlined,commandnameprefix=always]{changes}

\definechangesauthor[color=blue]{YG}

\makeatletter
\numberwithin{figure}{section}
\theoremstyle{plain}
\newtheorem{thm}{\protect\theoremname}[section]
\theoremstyle{definition}
\newtheorem{rem}[thm]{\protect\remarkname}
\theoremstyle{definition}
\newtheorem{defn}[thm]{\protect\definitionname}
\theoremstyle{plain}
\newtheorem{prop}[thm]{\protect\propositionname}
\theoremstyle{plain}
\newtheorem{lem}[thm]{\protect\lemmaname}
\theoremstyle{plain}

\theoremstyle{plain}
\newtheorem{cor}[thm]{\protect\corollaryname}

\theoremstyle{definition}

\theoremstyle{definition}

\theoremstyle{definition}

\theoremstyle{definition}

\usepackage{tikz}
\usetikzlibrary{shapes,arrows}
\usepackage{verbatim}
\usepackage{amsthm}
\usepackage{amstext}

\DeclareMathOperator{\diam}{diam}
\DeclareMathOperator{\dist}{dist}

\DeclareMathOperator{\Leb}{Leb}

\DeclareMathOperator{\supp}{supp}

\newcommand{\R}{\mathbb R}

\newcommand{\N}{\mathbb N}

\newcommand{\PP}{\mathbb P}

\newcommand{\eps}{\varepsilon}

\newcommand{\mH}{\mathcal{H}}

\newcommand{\udim}{\overline{\dim}_B\,}
\newcommand{\ldim}{\underline{\dim}_B\,}

\newcommand{\hdim}{\dim_H}
\newcommand{\lhdim}{\underline{\dim}_H}

\newcommand{\adim}{\dim_A}

\DeclareMathOperator{\Ker}{Ker}

\DeclareMathOperator{\Lip}{Lip}

\DeclareMathOperator{\Lin}{Lin}
\DeclareMathOperator{\Gr}{Gr}
\DeclareMathOperator{\Span}{span}

\makeatother

\providecommand{\conjecturename}{Conjecture}
\providecommand{\corollaryname}{Corollary}
\providecommand{\definitionname}{Definition}
\providecommand{\examplename}{Example}
\providecommand{\lemmaname}{Lemma}
\providecommand{\problemname}{Problem}
\providecommand{\propositionname}{Proposition}
\providecommand{\remarkname}{Remark}
\providecommand{\theoremname}{Theorem}
\providecommand{\taskname}{Task}

\def\vphi{\varphi}

\newcommand{\om}{\omega}

\def\N{{\mathbb N}}

\def\Jk{{\mathcal J}}

\def\be{\begin{equation}}
	\def\ee{\end{equation}}

\author[K. Bara\'{n}ski]{Krzysztof Bara\'{n}ski$^*$}
\address{$^*$Institute of Mathematics, University of Warsaw, ul.~Banacha 2, 02-097 Warszawa, Poland}
\email{baranski@mimuw.edu.pl}

\author[Y. Gutman]{Yonatan Gutman$^\dagger$}
\address{$^\dagger$Institute of Mathematics, Polish Academy of Sciences,
	ul.~\'Sniadeckich 8, 00-656 Warszawa, Poland}
\email{gutman@impan.pl}

\author[A. \'{S}piewak]{Adam \'{S}piewak$^\dagger$$^\S$}
\address{$^\S$Department of Mathematics, Bar-Ilan University, Ramat Gan, 5290002, Israel}
\email{ad.spiewak@gmail.com}

\title[Regularity of almost-surely injective projections]{Regularity of almost-surely injective projections in Euclidean spaces}
\date{\today}
\begin{document}

\begin{abstract} In a previous work we proved that if a finite Borel measure $\mu$ in a Euclidean space has Hausdorff dimension smaller than a positive integer $k$, then the orthogonal projection onto almost every $k$-dimensional linear subspace is injective on a set of full $\mu$-measure. In this paper we study the regularity of the inverses of these projections and prove that if $\mu$ has a compact support $X$ such that (respectively) the Hausdorff, upper box-counting or Assouad dimension of $X$ is smaller than $k$, then the inverse is (respectively) continuous, pointwise $\alpha$-H\"older for some $\alpha \in (0,1)$ or pointwise $\alpha$-H\"older for every $\alpha \in (0,1)$. The results generalize to the case of typical linear perturbations of Lipschitz maps and strengthen previously known ones in the lossless analog compression literature. We provide examples showing the sharpness of the statements. Additionally, we construct a non-trivial measure on the plane which admits almost-surely injective projections in every direction, and show that no homogeneous self-similar measure has this property. 
\end{abstract}

\maketitle

\section{Introduction and main results}

\subsection{Projections of sets and measures in Euclidean spaces}\label{subsec:proj}

The study of geometric and dimensional properties of the images of a set $X \subset \R^N$, $N \in \N$, under orthogonal projections
\[
P_V\colon \R^N \to V 
\]
onto $k$-dimensional linear spaces $V \subset \R^N$, is a well-known subject of interest in geometric measure theory. The space of all $k$-dimensional linear subspaces of $\R^N$ (or, equivalently, the space of corresponding orthogonal projections) forms the Grassmannian $\Gr(k, N)$, which has a structure of a $k(N-k)$-dimensional manifold, equipped with the standard rotation-invariant (Haar) measure. Throughout the paper, the terms `almost every linear space' or `almost every projection' will be used in relation to this measure.

A classical result in this area is the celebrated Marstrand--Mattila theorem, proved in \cite{Marstrand,Mattila-proj}. 

\begin{thm}[{\bf Marstrand--Mattila Projection Theorem}{}]\label{thm:marstrand-sets}Let $X$ be a Borel set in $\R^N$. Then the following hold.
\begin{enumerate}[$($i$)$]
\item $\hdim P_V (X) = \min\{k, \hdim X\}$ for almost every $k$-dimensional linear subspace $V$ of $\R^N$.
\item If $\hdim X > k$, then $P_V (X)$ has positive $k$-dimensional Hausdorff measure for almost every $k$-dimensional linear subspace $V$ of $\R^N$. 
\end{enumerate}
\end{thm}

Here and in the sequel $\dim_H X$ denotes the Hausdorff dimension of the set $X$, while $\mH^k$ denotes the $k$-dimensional Hausdorff measure. 

A number of further results related to Marstrand--Mattila projection theorem have been obtained, including versions valid for various types of dimension, and estimates on the size of the set of exceptional projections, see e.g.~\cite{FH97,PS00,Mattila04,B10,FO14,FO17,O21,dabrowski2021integrability} and the references therein, as well as \cite{FFJSixty} for a comprehensive survey. In particular, a version of Marstrand--Mattila's projection theorem for measures has been established (see \cite{HuTaylorProjections, HuntKaloshin97}). 

\begin{thm}[{\bf Marstrand--Mattila Projection Theorem for measures}{}]\label{thm:marstrand}
Let $\mu$ be a finite Borel measure in $\R^N$. Then the following hold.
\begin{enumerate}[$($i$)$]
\item $\hdim P_V \mu = \min\{k, \hdim \mu\}$ and $\lhdim P_V \mu = \min\{k, \lhdim \mu\}$ for almost every $k$-dimen\-sional linear subspace $V$ of $\R^N$.
\item If $\lhdim \mu > k$, then $P_V \mu$ is absolutely continuous with respect to  $k$-dimensional Hausdorff measure for almost every $k$-dimensional linear subspace $V$ of $\R^N$.
\end{enumerate}
\end{thm}

Here $P_V \mu$ denotes the image of $\mu$ under $P_V$, while $\hdim \mu$ and $\lhdim \mu$ denote, respectively, the upper and lower Hausdorff dimensions of a measure $\mu$ (see Section~\ref{defn:dim} for the definitions).

\subsection{Injective and almost-surely injective projections}\label{subsec:inj_proj}
 
Apart from considering the dimension of the images of $X$ under orthogonal projections $P_V$, another line of research is to study under which conditions the projections $P_V$ are injective on $X$, at least for typical $V$.\footnote{Another question one may consider is when a projection $P_V$ is injective on $X$ for  \textit{some}  $V$. In \cite[Example~III.4]{LosslessAnalogCompression} one finds a construction of a compactly supported measure $\mu$ in $\R^3$ with $\hdim \mu =2$, such that there is a projection $P\colon \R^3\rightarrow \R$ injective on $\supp \mu$.} Note that if this occurs, then $P_V$ provides a \emph{topological embedding} of $X$ into a $k$-dimensional linear space $V$, and $X$ can be considered as the graph of a function from $P_V(X) \subset V \simeq \R^k$ to $V^\perp \simeq \R^{N-k}$.

It is known that if $X$ is a compact subset of $\R^N$ and $2\udim X < k$, where $\udim$ denotes the upper box-counting (Minkowski) dimension, then $P_V$ is injective for a typical $k$-dimensional linear space $V \subset \R^N$. This fact is commonly referred to as the \emph{Ma\~{n}\'e projection theorem}. Indeed, Ma\~{n}\'e proved this result for topologically generic projections in \cite{Mane81}, while a version valid for almost every projection (or, equivalently, for almost every linear map $L \colon \R^N \to \R^k$) was presented in \cite{SYC91,Rob11}. Notice that this agrees with the well-known Menger--N\"obeling embedding theorem (see e.g.~\cite[Theorem~5.2]{HW41}, which states that for a compact metric space $X$ with Lebesgue covering dimension at most $k$, a generic continuous transformation $\phi\colon X \to \R^{2k+1}$ is injective.

Given a projection $P_V$ which is injective on a set $X$, it is natural to ask what is the regularity of its \emph{inverse}
\[
(P_V|_X)^{-1} \colon P_V(X) \to X.
\]
In \cite[Theorem~3.1]{HK99}, Hunt and Kaloshin proved that if $X \subset \R^N$ is compact and $2\udim X < k$, then for almost every $k$-dimensional linear space $V \subset \R^N$, the projection $P_V$ restricted to $X$ has an $\alpha$-H\"older inverse for $0 < \alpha < 1 - \frac{2 \udim X}{k}$ (see also \cite{BAEFN93,EFNT94} for earlier results in this direction and \cite{Rob11} for a detailed exposition). In \cite[Theorem~2.1]{RossiShmerkinHolderCoverings}, Rossi and Shmerkin gave upper bounds on the Hausdorff dimension of the set of exceptional projections. Furthermore, the regularity of the inverses can be improved, if the assumption $2\udim X < k$ is replaced by $2\adim X < k$, where $\adim$ is the Assouad dimension (see Definition~\ref{defn:adim}). More precisely, in this case almost all projections onto $k$-dimensional linear subspaces of $\R^N$ have inverses which are $\alpha$-H\"older for any $\alpha \in (0,1)$, see \cite[Theorem~5.2]{olson2002bouligand} and \cite[Theorem~9.18]{Rob11}. The problem of the existence of linear embeddings and the regularity of their inverses was also studied for finite-dimensional subsets of Banach spaces, see \cite[Chapters~5--9]{Rob11} and references therein.

It is known that in general, the bound $2\udim X < k$ in the Ma\~{n}\'e projection theorem cannot be diminished (see \cite[Example~V.3]{HW41}), and $\udim$ cannot be replaced by $\hdim$ (see \cite[Appendix]{SYC91})\footnote{In fact, these examples show that even the existence of an orthogonal projection injective on $X$ does not hold under weaker assumptions.}. However, the situation changes if instead of the injectivity of $P_V$ on $X$, one is interested in \emph{almost sure injectivity} of $P_V$, i.e.~the injectivity of $P_V$ on a full $\mu$-measure Borel subset of $X$, according to a given Borel measure $\mu$ on $X$. In our previous paper \cite[Corollary~3.4]{BGS20}, strengthening a result by Alberti, B\"olcskei, De Lellis, Koliander and Riegler \cite{LosslessAnalogCompression}, we showed the following.

\begin{thm}[{\bf Probabilistic injective projection theorem}{}]
Let $X$ be a Borel subset of $\R^N$ equipped with a finite or $\sigma$-finite Borel measure $\mu$ and assume $\mH^k(X) = 0$ for some positive integer $k \le N$
$($in particular, it is enough to assume $\hdim X < k)$. Then for almost every $k$-dimensional linear subspace $V$ of $\R^N$, the orthogonal projection $P_V$ is injective on a full $\mu$-measure Borel set $X_V \subset X$.
\end{thm}

See also Theorem~\ref{thm: hdim injectivity} for a more general statement. Consequently, for a given set $X$, the minimal \emph{embedding dimension} $k$ for a typical projection can be reduced by half when considered in a `probabilistic' setup compared to a `deterministic' one. 
 
We note that the condition $\hdim \mu \leq k$ is necessary for a typical almost sure injectivity of $P_V$. Indeed, we have the following.

\begin{prop}\label{prop: dim bound}
Let $\mu$ be a finite Borel measure in $\R^N$ with $\hdim \mu > k$. Then for almost every $k$-dimensional linear space $V \subset \R^N$ and every Borel set $Y \subset \R^N$ of full $\mu$-measure, the orthogonal projection $P_V$ is not injective on $Y$.
\end{prop}

The proof of Proposition~\ref{prop: dim bound} is presented in Section~\ref{sec:projections}.

\begin{rem}
Note that one cannot extend Proposition~\ref{prop: dim bound} to projections onto all $k$-dimensional linear subspaces $V \subset \R^N$. An obvious counterexample is the lift of $1$-dimensional Lebesgue measure on the interval $[0,1]$ in the $x$-axis on the plane to a graph of a Borel function $[0,1] \to \R$ with Hausdorff dimension greater than $1$ (e.g.~the graph of a continuous nowhere-differentiable Weierstrass-type function), which projects injectively onto the $x$-axis. On the other hand, if $\mu$ is $s$-\emph{analytic} for $s > k$, then no projection onto a $k$-dimensional linear space $V \subset \R^N$ is injective on a set of positive $\mu$-measure, see \cite[Corollary~IV.2]{LosslessAnalogCompression}.
\end{rem}

In the border case $\hdim \mu = k$, if $\mu$ is not singular with respect to the $k$-dimensional Hausdorff measure, then different types of behaviour may occur, e.g.~for $k=1$ and $1$-dimensional Hausdorff measure on an interval in $\R^2$, a projection onto a typical line is injective on its support, while for $1$-dimensional Hausdorff measure on a circle in $\R^2$, no Lipschitz map $\phi \colon  \R^2 \to \R$ is injective on a set of full measure, see \cite[Example~3.5]{BGS20}. Understanding the border case is closely related to the following open problem in geometric measure theory: if $X \subset \R^2$ is a $\mH^1$-measurable, purely unrectifiable set with $\mH^1(X)<\infty$, is it true that for $\mH^1$-almost every $x \in X$, almost every line passing through $x$ meets $X$ only at $x$? Refer to \cite[Problem~12]{Mattila04} for the original formulation and to \cite[Conjecture~4.15]{MattilRectifSurvey2023} for a more recent account. While the problem   remains unsolved in the general case, it has been answered affirmatively for some classes of sets, e.g. for self-similar 1-sets satisfying the Open Set Condition, see \cite{SimonSolomyakVisibility}.

In this paper, assuming suitable bounds on the dimension of a compact set $X \subset \R^N$, we study the question of the regularity of the \emph{almost sure inverse} map
\[
(P_V|_{X_V})^{-1} \colon P_V(X_V) \to X_V
\]
for a typical $k$-dimensional linear space $V \subset \R^N$, where $P_V$ is injective on a full $\mu$-measure Borel set $X_V \subset X$. 

\begin{defn} A map $\phi\colon A \to \R^N$, where $A \subset \R^k$, is \emph{pointwise} $\alpha$-\emph{H\"older}, if for every $x \in A$ there exists $c_x > 0$ such that 
\[
\|\phi(x)-\phi(y)\| \le c_x \|x-y\|^\alpha
\]
for every $y \in A$.
\end{defn}

The basic result of this paper is the following.

\begin{thm}[{\bf Regularity of the inverse of almost-surely injective projections}{}]
\label{thm:injectivity-proj}
Let $X$ be a compact subset of $\R^N$ and let $\mu$ be a finite Borel  measure supported on $X$. Consider orthogonal projections $P_V\colon \R^N \to V$ onto $k$-dimensional linear spaces $V \subset \R^N$. Then the following hold.
	\begin{enumerate}[$($i$)$]
	
	\item\label{it:hdim proj}
	If $\hdim X < k$, then for almost every $V$ there exists a full $\mu$-measure Borel set  $X_V \subset X$ such that the restriction of $P_V$ to $X_V$ is injective with continuous inverse.
	
	\item\label{it:bdim proj}
	If $\udim X < k$, then for almost every $V$ there exists a full $\mu$-measure Borel set  $X_V \subset X$ such that the restriction of $P_V$ to $X_V$ is injective with inverse which is pointwise $\alpha$-H\"{o}lder for every $\alpha \in \big(0,1 -  \frac{\udim X}{k}\big)$.
	
	\item\label{it:adim proj}
	If $\adim X < k$, then for almost every $V$ there exists a full $\mu$-measure Borel set  $X_V \subset X$ such that the restriction of $P_V$ to $X_V$ is injective with inverse which is pointwise $\alpha$-H\"{o}lder for every $\alpha \in (0,1)$.
	
	\end{enumerate}
\end{thm} 

Theorem~\ref{thm:injectivity-proj} follows from a more general Theorem~\ref{thm:inverses main}, presented in Subsection~\ref{subsec:lipschitz}. In Section~\ref{sec:examples} we provide examples showing that the result is sharp in several ways. In particular, the bound $1 - \udim X / k$ for the H\"{o}lder exponent in assertion~(ii) cannot be improved in terms of $\udim X$.

\begin{rem}
It is important to note that, in general, we cannot obtain H\"older continuity (instead of pointwise H\"older continuity) of the inverse of $P_V|_{X_V}$ in Theorem~\ref{thm:injectivity-proj}, even under the assumption $\adim X < k$. Indeed, if $X \subset \R^N$ is a compact set with $\adim X < k$, which does not embed topologically into $\R^k$ (for example, there exist simplicial complexes of dimension $n$ which do not embed topologically into $\R^{2n}$, see \cite[Example~V.3]{HW41}) and $\mu$ is any measure with $\supp\mu = X$, then a projection $P_V$ onto a $k$-dimensional linear space $V \subset \R^N$, which is injective on a set of full $\mu$-measure with a H\"older inverse is actually a homeomorphism on $X$, which is impossible, as $X$ does not embed topologically into $\R^k$. See \cite[Remark~3.10]{BGS20} for the details of this argument.
\end{rem}

\subsection{{Measures with almost-surely injective projections in every direction}}

An important question in the geometric measure theory, which has gained increasing interest in recent years, is finding conditions under which suitable projection theorems hold for \textit{every} (rather than almost every) projection. This question is particularly interesting in the context of
\emph{self-similar measures} for iterated function systems, i.e.~Borel probability measures $\mu$ in $\R^N$ satisfying
\[ \mu = \sum \limits_{i\in I} p_i\: \vphi_i \mu,\]
where $I$ is a finite set, $(p_i)_{i \in I}$ is a strictly positive probability vector and $\{ \vphi_i : i \in I \}$ is an iterated function system (IFS) consisting of contracting similarities
\[
\vphi_i \colon \R^N \to \R^N, \qquad \vphi_i(x) = r_i O_i(x) + t_i 
\]
with scales $r_i \in (0,1)$, orthogonal matrices $O_i \in \R^{N \times N}$ and translation vectors $t_i \in \R^N$.

In \cite[Theorem~1.6]{HochmanShmerkinLocalEntropy}, Hochman and Shmerkin proved that if the system $\{\vphi_i : i \in I \}$ satisfies the Strong Separation Condition and the semigroup generated by $\{O_i : i \in I\}$ acts minimally on $\Gr(k,N)$, then the first assertion of Marstrand--Mattila's projection theorem for measures (Theorem~\ref{thm:marstrand}) holds for all orthogonal projections onto $k$-dimensional linear subspaces of $\R^N$. The result was extended by Farkas \cite[Theorem~1.6]{FarkasProjections16} for systems without any separation conditions. On the other hand, Rapaport \cite{Rapaport17} constructed a self-similar measure in the same class, which does not satisfy the absolute continuity part of Theorem~\ref{thm:marstrand} for a Baire-residual set of projections. Similarly, he showed that for such measures, the Slicing theorem (Theorem~\ref{thm:sections}) can fail for a residual set of projections. Nevertheless, in \cite{Rapaport20} he proved  that typical \emph{homogeneous} self-similar measures (i.e.~the ones satisfying $r_i O_i = r_j O_j$ for all $i, j$) in the plane with the Strong Separation Condition satisfy both assertions of Theorem~\ref{thm:marstrand} for every projection. In the case of random constructions, Simon and Rams \cite{RamsSimonPercolations14,RamsSimonPercolations15} showed that almost every fractal percolation satisfies the assertions of Theorem~\ref{thm:marstrand-sets} for every projection.

Considering the question of the injectivity of projections, we note that, obviously, a non-singleton set in $\R^N$ cannot be projected injectively onto all $k$-dimensional linear subspaces of $\R^N$, for any $1 \leq k \le N-1$.
This naturally leads us to the setup of almost-surely injective projections.
The first question appearing in this context is whether there exist non-trivial (with non-singleton support) measures which project almost-surely injectively in every direction. The following result shows that in fact such measures exist.

\begin{thm}[{\bf Existence of measure with almost-surely injective projections in every direction}{}]
\label{thm:every direction a.s. injectivity}
There exists a compactly supported Borel probability measure $\mu$ in $\R^2$ with positive Hausdorff dimension $($and hence with non-singleton support$)$, such that every orthogonal projection $P_V\colon \R^2 \to V$ onto a line $V \subset \R^2$ is injective on a Borel set $X_V$ of full $\mu$-measure.
\end{thm}

The proof of Theorem~\ref{thm:every direction a.s. injectivity} is presented in Section~\ref{sec:every_direction}. 

Furthermore, we show that, unlike for the Marstrand--Mattila and Slicing theorems, non-dege\-ne\-rated homogeneous self-similar measures cannot satisfy this property, even generically.

\begin{prop}\label{prop:self-sim}
Let $\mu$ be a self-similar measure in $\R^N$ corresponding to a homogeneous IFS $\vphi_i(x) = rO(x) + t_i,\ i \in I$, such that $\vphi_i$ are not all equal. Then for every $k \in \{1,\ldots, N-1\}$ there exists a $k$-dimensional linear space $V \subset \R^N$ such that for every Borel set $Y \subset \R^N$ of full $\mu$-measure, the orthogonal projection $P_V$ onto $V$ is not injective on $Y$.
\end{prop}

Proposition~\ref{prop:self-sim} follows from a more general fact (Proposition~\ref{prop:ifs with translate}), which is proved in Section~\ref{sec:every_direction}.

\subsection{Regularity of the inverse for typical linear perturbations of Lipschitz maps}\label{subsec:lipschitz}

The result described in Theorem~\ref{thm:injectivity-proj} can be generalized to the setup of almost-surely injective linear perturbations of Lipschitz maps on compact sets in Euclidean spaces. In a previous paper, we proved the following.

\begin{thm}[{\cite[Theorem~3.1]{BGS20}}]\label{thm: hdim injectivity}
Let $\mu$ be a finite or $\sigma$-finite Borel measure in $\R^N$ supported on a set $X$, such that $\mu$ is singular with respect to the $k$-dimensional Hausdorff measure for some $k \in \N$ $($in particular, it is enough to assume $\hdim \mu < k)$, and let $\phi \colon X \to \R^k$ be a Lipschitz map. Then for almost every linear transformation $L\colon \R^N \to \R^k$ there exists a Borel set $X_L \subset X$ of full $\mu$-measure, such that the map $\phi_L = \phi + L$ is injective on $X_L$.
\end{thm}

Here and in the sequel, `almost every linear map' refers to the Lebesgue measure in the space $\Lin(\R^N, \R^k) \simeq \R^{Nk}$.
The result can be generalized to the case of H\"older maps $\phi\colon X \to \R^k$, see \cite{BGS20} for details.
The conclusion of Theorem~\ref{thm: hdim injectivity} holds also for prevalent sets in the spaces of Lipschitz and $C^r$-maps $\phi\colon X \to \R^k$ (see Definition~\ref{def:prevalence} and Remark~\ref{rem:prevalence lipschitz}).

Extending the setup of Subsection~\ref{subsec:inj_proj}, we study regularity properties of the inverse maps 
\[
(\phi_L|_{X_L})^{-1} \colon \phi_L(X_L) \to X_L.
\]

The following is our main result.

\begin{thm}[{\bf Regularity of the inverse for almost-surely injective linear perturbations of Lipschitz maps}{}]
\label{thm:inverses main}
	Let $\mu$ be a finite Borel measure in $\R^N$ with a compact support $X$ and let $\phi \colon X \to \R^k$ be a Lipschitz map. We write $\phi_L = \phi  + L$ for linear maps $L \in \Lin(\R^N, \R^k)$. Then the following hold.
	\begin{enumerate}[$($i$)$]
		\item\label{it:hdim main} If $\hdim X < k$, then for almost every linear map $L$ there exists a Borel set $X_L \subset X$ of full $\mu$-measure such that $\phi_L$ is injective on $X_L$ with continuous inverse. 

		\item\label{it:udim main} If $\udim X <k$, then for almost every linear map $L$ there exists a Borel set $X_L \subset X$ of full $\mu$-measure such that $\phi_L$ is injective on $X_L$ with inverse which is pointwise $\alpha$-H\"{o}lder for every $\alpha \in \big(0,1 -  \frac{\udim X}{k}\big)$.

		\item\label{it:adim main} If $\adim X<k$, then for almost every linear map $L$ there exists a Borel set $X_L \subset X$ of full $\mu$-measure on which $\phi_L$ is injective with inverse which is pointwise $\alpha$-H\"{o}lder for every $\alpha \in (0,1)$.
	\end{enumerate}
\end{thm} 

Instead of linear maps from $\Lin(\R^N, \R^k) \simeq \R^{Nk}$, $k \le N$, with the Lebesgue measure, one can equivalently consider orthogonal projections $P_V\colon \R^N \to V$ for $V \in \Gr(k, N)$ with the rotation-invariant measure (see Remark~\ref{rem:linear to projection}). Therefore, Theorem~\ref{thm:inverses main} implies immediately Theorem~\ref{thm:injectivity-proj} by setting $\phi=0$.  On the other hand, considering $\phi$ within the spaces of Lipschitz or $C^r$-maps, $r = 1, 2, \ldots, \infty$, one obtains a result for almost every map in these spaces in the sense of prevalence, as defined in \cite{Prevalence92}. See Definition~\ref{def:prevalence} and Remark~\ref{rem:prevalence lipschitz} for more details. Examples presented in Section~\ref{sec:examples} show that Theorem~\ref{thm:inverses main} cannot be improved in the setting of prevalence within the above spaces.

\subsection{Relation to the theory of compressed sensing}

The field of \textit{compressed sensing} grew out of the work of Candès, Donoho, Romberg and Tao (\cite{candes2006compressive,candes2006stable,donoho2006compressed,FR13}).
The fundamental problem of the theory is to find conditions enabling to recover an input vector $x\in \R^{N}$ from its linear measurement $y=A x\in \R^{m}$, where $A\in\R^{m\times N}$, even though $m\ll N$.  A key theorem (\cite[Theorem~9.12]{FR13}, see also \cite{candes2006near, CandesRombergTao}) states that
with high probability one may recover $x$ with $\|x\|_0:=|\{j:\,x_j\neq 0\}|\leq s$ (the \textit{$s$-sparsity} condition) from $y$ when $A\in\R^{m\times N}$ is a random Gaussian matrix with $m\approx s\ln\frac{N}{s}$ via an $\ell_1$-minimization \textit{basis pursuit} algorithm \cite[$\S$1.4.3]{mallat1999wavelet}, see also \cite{chen1995atomic}.
Capitalizing on sparsity, compressed sensing has found many applications (see e.g. \cite{lustig2007sparse,duarte2008single,baraniuk2007compressive,herman2009high}).

In \cite{LosslessAnalogCompression}, Alberti, B\"{o}lcskei, De Lellis, Koliander and Riegler studied the  above-mentioned problem in a setting where both the input vector $x$ and the sensing matrix $A$ are random:  $x\in \R^{N}$ is given according to a probability measure $\mu$ in $\R^n$, $A\in\R^{m\times N}$ is given according to the Lebesgue measure in $\R^{m\times N}$ and one seeks to recover $x$ from $y=Ax$ $\mu$-almost surely.  In particular, they proved a version of Theorem~\ref{thm: hdim injectivity} for probability measures, with $\phi = 0$ and the Hausdorff dimension replaced with the lower modified Minkowski dimension, see \cite[Theorem~II.1]{LosslessAnalogCompression} (for related earlier results see \cite[Theorem~18 and Corollary~1]{WV10}).

In recent years there has been a surge of interest in a compressed sensing framework for analog signals modelled by continuous-alphabet discrete-time stochastic processes\footnote{The rigorous passage between continuous-time signals and discrete-time signals is justified by the Shannon sampling theorem (\cite[Chapter~1]{higgins1996sampling}).} with general (not necessarily sparse) distributions (\cite{WV10,donoho2010precise,donoho2011noise,jalali2017universal,CompressionCS17, geiger2019information,GSNewBounds,mmdimcompress}). In \cite{WV10}, Wu and Verd\'u emphasized that the regularity of both of the encoder and decoder is crucial, as it introduces resilience to noise. 
Translated to the setting of \cite{LosslessAnalogCompression} and Theorem~\ref{thm: hdim injectivity}, as the encoder is already assumed to be linear, this corresponds to investigating the regularity of the inverse map $L^{-1}\colon L(X_L)\rightarrow X_L$. Thus, one may interpret our main result as giving almost sure regularity guarantees for decompression under various dimension assumptions on the measure generating the input vector. For example, Theorem~\ref{thm:inverses main} allows us to improve \cite[Corollary~1]{WV10}. According to this result,
denoting by $\Sigma_s$  the set of all $s$-sparse vectors in $\R^N$, i.e.
\[ \Sigma_s = \{ x  \in \R^N :  x = (x_1, \ldots, x_N) \text{ with at most } s \text{ of } x_i \text{ being non-zero} \}, \]
\noindent
if $\mu$ is a $\sigma$-finite Borel measure on $\Sigma_s$, $k > s$, $\eps>0$ and $0<\alpha<1-\frac{s}{k}$, then there exists a linear map $L \colon \R^N \to \R^k$ and  an $\alpha$-H\"older map $g \colon \R^k \to \R^N$ such that $\mu \left( \left\{ x \in \R^N : g(Lx) \neq x \right\} \right) \leq \eps$.

As $\Sigma_s$ is a finite union of $s$-linear subspaces of $\R^N$, one has $\adim \Sigma_s =s$ and thus Theorem~\ref{thm:inverses main}(iii) implies in a straightforward manner the following result, where the H\"older exponent $\alpha$ can be taken arbitrarily close to $1$.

\begin{cor}

Let $\mu$ be any compactly supported finite Borel measure on $\Sigma_s$ and fix $k > s$. Then almost every linear map $L \colon  \R^N \to \R^k$ is injective on a set of full $\mu$-measure with an inverse which is pointwise $\alpha$-H\"older for every $\alpha \in (0,1)$. Consequently, for almost every linear map $L \colon  \R^N \to \R^k$, every $\eps > 0$ and every $\alpha \in (0,1)$, there exists an $\alpha$-H\"older map $g \colon  \R^k \to \R^N$ such that $\mu \left( \left\{ x \in \R^N : g(Lx) \neq x \right\} \right) \leq \eps$.
\end{cor}

In fact, we can guarantee $g$ to be Lipschitz up to a logarithmic factor, see Theorem~\ref{thm:assouad full}. The result extends to the setup of compactly supported measures on a finite union of $s$-dimensional $C^1$ manifolds. For related results see \cite{stotz2017almost}.

The article \cite{LosslessAnalogCompression} emphasizes the importance of proving \textit{converse statements} for compression results, i.e.~the ones demonstrating that if the dimension of the image space is too small, then compression is not possible, and  provides such a result for \textit{s-analytic} measures, see \cite[Corollary~IV.2]{LosslessAnalogCompression}. Proposition~\ref{prop: dim bound} provides a converse statement to arbitrary finite Borel measures  in $\R^N$, in terms
of their Hausdorff dimension.

\subsection{Further related topics}
Let us mention that almost sure injectivity plays an important role in some other nonlinear projection schemes. For instance, in the context of natural projection maps from the symbolic space for iterated function systems, this property is called \textit{weak almost unique coding} (see \cite[Definition~1.9]{KolossvarySimonGL}) and for self-similar systems it is known to be equivalent to the \emph{no dimension drop} condition (i.e.~the equality of the dimension of given ergodic invariant measure to the ratio of its entropy and Lyapunov exponent), see \cite[Appendix]{KolossvarySimonGL} and \cite[Corollary~4.7]{FengDimensionSelfAffine}. This observation can be successfully utilized for obtaining dimension results for certain classes of fractal attractors, see e.g. \cite{KolossvarySimonGL}. Moreover, basic techniques developed for studying typical properties of orthogonal projections can be often transferred to parametrized families of iterated functions systems satisfying the \textit{transversality condition} (an analogue of Lemma~\ref{lem: key_ineq_linear}). It was first used by Pollicott and Simon \cite{PollicottSimonDeleted} and led, for example, to results analogous to Marstrand's projection theorem (\cite{SSUPacific,SSUTrans}, see also \cite{BSSS22} for a more detailed overview and \cite{SolomyakTransversSurvey} for a recent survey).

Finally, techniques originating from the study of orthogonal projections can also be used to analyse so-called \textit{delay-coordinate maps}, i.e.~maps of the form $\phi (x) = (h(x), h(Tx), \ldots, h(T^{k-1}x))$, where $T \colon X \to X$ is a discrete-time dynamical system and $h \colon X \to \R$ is a real function (observable). The injectivity or almost sure injectivity of $\phi$ implies that the original dynamics can be faithfully modelled based only on the values of the observable. Related embedding results, know as Takens-type theorems, serve as a framework for applications in natural sciences see e.g.~\cite{T81, sm90nonlinear, SYC91, Abarbanel_book}. Recently, a probabilistic counterpart of this theory has been developed \cite{SSOY98, BGS20,BGS22,BGSPredict}, where the regularity of almost-surely defined mappings related to $\phi$ plays a crucial role. See \cite{BGSPredict} for a more detailed discussion and further references.

\subsection*{Structure of the paper}

Section~\ref{sec:prelim} contains basic definitions and a description of technical tools used in subsequent parts of the paper. In Section~\ref{sec:projections} we prove Proposition~\ref{prop: dim bound}, while Section~\ref{sec:proof_main} contains the proof of Theorem~\ref{thm:inverses main}. Theorem~\ref{thm:every direction a.s. injectivity} and Proposition~\ref{prop:self-sim} are proved in Section~\ref{sec:every_direction}. The last Section~\ref{sec:examples} provides examples showing the sharpness of the obtained results.

\section{Preliminaries}\label{sec:prelim}

We consider the Euclidean space $\R^N$, $N \ge 1$, with the standard inner product $\langle \cdot, \cdot \rangle$. The corresponding norm and diameter are denoted, respectively, by $\| \cdot \|$ and $|\cdot |$. The symbol $B_N(x,r)$ denotes the $r$-ball centred at $x$ in the Euclidean norm in $\R^N$. We often write $B(x,r)$ when the dimension is clear from the context. By $\# A$ we denote the cardinality of a set $A$.

\subsection{Dimensions}\label{subsec:dim}

\begin{defn}\label{defn:dim}
	
	For $s>0$, the \emph{$s$-dimensional $($outer$)$ Hausdorff measure} of a set $X \subset \R^N$ is defined  as
	\[ \mH^s(X) = \lim \limits_{\delta \to 0}\ \inf \Big\{ \sum \limits_{i = 1}^{\infty} |U_i|^s : X \subset \bigcup \limits_{i=1}^{\infty} U_i,\ |U_i| \leq \delta  \Big\}.\]
	The \emph{Hausdorff dimension} of $X$ is given as
	\[ \hdim X = \inf \{ s > 0 : \mathcal{H}^s(X) = 0 \} = \sup \{ s > 0 : \mathcal{H}^s(X) = \infty \}. \]
	For a bounded set $X \subset \R^N$ and $\delta>0$, let $N(X, \delta)$ denote the minimal number of balls of radius $\delta$ required to cover $X$. The \emph{lower} and \emph{upper box-counting $($Minkowski$)$ dimensions} of $X$ are defined, respectively, as
	\[ \ldim X = \liminf \limits_{\delta \to 0} \frac{\log N(X,\delta)}{-\log \delta,}\qquad \text{ and }\quad \udim X = \limsup \limits_{\delta \to 0} \frac{\log N(X,\delta)}{-\log \delta}. \]
For a finite Borel measure $\mu$ in $\R^N$, we define its (upper) Hausdorff dimension as
\[ \hdim\mu = \inf \left\{ \hdim X : X \subset \R^N \text{ Borel with } \mu(\R^N \setminus X) = 0 \right\} \]
and the lower Hausdorff dimension as
\[ \lhdim\mu = \inf \left\{ \hdim X  : X \subset \R^N \text{ Borel with } \mu(X) > 0 \right\}. \]
\end{defn}

\begin{defn}\label{defn:adim}
A bounded set $X \subset \R^N$ is said to be $(M,s)$-\emph{homogeneous} if $N(X \cap B(x, r), \rho) \leq M(r / \rho)^s$ for every $x \in X$, $0 < \rho < r$, i.e.~the intersection $B(x,r) \cap X$ can be covered by at most $M(r / \rho)^s$ balls of radius $\rho$. The \emph{Assouad dimension} of $X$ is defined as
\[ \adim X = \inf \{ s > 0 : X \text{ is } (M, s)\text{-homogeneous for some } M > 0 \}. \]
\end{defn}
It is easy to see that in the definitions of box-counting and Assouad dimensions, it is enough to consider covers by balls centred in $X$. For a bounded set $X \subset \R^N$ and a finite Borel measure $\mu$ on $X$, we have the following inequalities (see e.g. \cite[(9.1)]{Rob11}).
	\begin{equation*}
		\lhdim \mu \leq \hdim \mu \leq \hdim X \leq \ldim X \leq \udim X \leq \adim X.
	\end{equation*}

\subsection{Prevalence}\label{subsec:prevalence}

A notion of \textit{prevalence} was introduced by Hunt, Shroer and Yorke in \cite{Prevalence92} and is regarded to be an analogue of `Lebesgue almost sure' condition in infinite dimensional linear spaces.

\begin{defn}\label{def:prevalence}
Let $\mathcal V$ be a complete linear metric space (i.e.~a linear space with a complete  metric which makes addition and scalar multiplication continuous). A Borel set $\mathcal S \subset \mathcal V$ is called \emph{prevalent} if there exists a Borel measure $\nu$ in $\mathcal V$, which is positive and finite on some compact set in $\mathcal V$, such that for every $v \in \mathcal V$, we have $v + e \in \mathcal S$ for $\nu$-almost every $e \in \mathcal V$. A non-Borel subset of $\mathcal V$ is prevalent if it contains a prevalent Borel subset. 
\end{defn} 

We focus mainly on the prevalence in the space $\Lip(X, \R^k)$ of all Lipschitz functions $h \colon X \to \R^k$ on a compact set $X \subset \R^N$, endowed with the Lipschitz norm 
\begin{equation}\label{eq:lip_norm}
\|h\|_{\Lip} = \|h\|_{\infty} + \Lip(h),
\end{equation}
where $\|h\|_\infty$ is the supremum norm and $\Lip(h)$ is the Lipschitz constant of $h$. Note, however, that in Theorem~\ref{thm:inverses main} we can consider prevalence in other spaces, as explained in the following remark.

\begin{rem}\label{rem:prevalence lipschitz}
Let $\mathcal V$ be any of the spaces of Lipschitz or $C^r, r= 1,2, \ldots, \infty$, maps from (a bounded open neighbourhood of) a compact set $X\subset \R^N$ into $\R^k$, endowed with the natural complete linear metric. Note that in order to show prevalence of a set $\mathcal S \subset \mathcal V$ via linear maps, it is enough to prove that for every $\phi \in \mathcal V$ and Lebesgue-almost every $L\in \Lin(\R^N, \R^k) \simeq \R^{Nk}$ we have $\phi + L \in \mathcal S$. As all the $C^r$-spaces listed above are contained in the space of Lipschitz maps, it is enough to consider Lipschitz maps $\phi$.
\end{rem}

\subsection{Conditional measures}

It will be useful to work with a system of conditional measures for considered projections. In the sequel, we will denote Dirac's delta measure at a point $x$ by $\delta_x$. The symbol $\phi\mu$ denotes the image of a measure $\mu$ under a map $\phi$.

\begin{defn}\label{def:conditional measures}
Given a continuous map $\phi\colon X \to \R^k$ on a compact set $X \subset \R^N$ there exists a \emph{system of conditional measures} of a probability measure $\mu$ on $X$ with respect to $\phi$, i.e.~a family $\{ \mu_y : y \in \R^k\}$, such that
\begin{enumerate}[$($i$)$]
	\item\label{item:cond_supp} for every $y \in \R^k,\ \mu_y$ is a $($possibly zero$)$ Borel measure on $\phi^{-1}(\{y\})$,
	\item for $\phi\mu$-almost every $y \in \R^k$, $\mu_y$ is a Borel probability measure,
	\item\label{item:cond_total_prob} for every $\mu$-measurable set $A \subset X$, the function $\R^k \ni y \mapsto \mu_y(A)$ is $\phi\mu$-measurable and
	\[ \mu(A) = \int_{\R^k} \mu_y(A)d\phi\mu(y). \]
\end{enumerate}
\end{defn}
The existence and $\phi \mu$-almost sure uniqueness of the system of conditional measures follows from the Rokhlin's Disintegration Theorem \cite{Rohlin52}. See also \cite{SimmonsRohlin} for a more direct approach.

The following lemma characterizes almost sure injectivity in terms of conditional measures. 

\begin{lem}\label{lem:as injectivity Diracs}
Let $\phi \colon X \to \R^k$ be a continuous map on a compact set $X \subset \R^k$ and let $\mu$ be a Borel probability measure on $X$. Then $\phi$ is injective on a Borel set $X_\phi \subset X$ of full $\mu$-measure if and only if the system $\{ \mu_{\phi(x)} : x \in X \}$ of conditional measures of $\mu$ with respect to $\phi$ satisfies $\mu_{\phi(x)} = \delta_x$ for $\mu$-almost every $x \in X$.
\end{lem}

\begin{proof}
If $\phi$ is injective on $X_\phi$, then setting $\mu_{\phi(x)} = \delta_x$ for $x \in X_\phi$ and $\mu_y = 0$ for $y \notin \phi(X_\phi)$ gives a system of conditional measures of $\mu$ with respect to $\phi$ (see \cite[p.~620]{BGS22} for a detailed argument). Hence, the first implication follows by the almost sure uniqueness of the system of conditional measures. For the other implication, assume  $\mu_{\phi(x)} = \delta_x$ for $\mu$-almost every $x \in X$. Then $X_\phi = \{ x \in X : \mu_{\phi(x)} = \delta_x\}$ is the required set of injectivity. Indeed, $\mu(X_\phi) = 1$ by assumption, and if $x, y \in X_\phi$ and $\phi(x) = \phi(y)$, then $\delta_x = \delta_y$, so $x=y$.
\end{proof}

\section{Orthogonal projections and slices}\label{sec:projections}
As noted in the introduction, we consider the Grassmannian $\Gr(k,N)$ of $k$-dimensional linear subspaces of $\R^N$. We denote by $\gamma_{k,N}$ the unique rotation-invariant measure on $\Gr(k,N)$ (see \cite[Section~3.9]{mattila} for details and \cite{FR02} for an alternative construction). Recall that for $V \in \Gr(k,N)$ we denote by $P_V \colon \R^N \to V \simeq  \R^k$ the orthogonal projection onto $V$.

\begin{rem}\label{rem:linear to projection}
As we switch between Lebesgue-almost sure statements for linear transformations $L \in \Lin(\R^N, \R^k) \simeq \R^{Nk}$ and $\gamma_{k,N}$-almost sure statements for orthogonal projections $P_V$, $V \in \Gr(k,N)$, it is useful to note that the two kinds of statements are equivalent if one is interested in the injectivity of $L$ and $P_V$ on a full-measure set. Namely, Lebesgue-almost every linear map $L \colon \R^N \to \R^k$, $k \le N$, has full rank and can be represented uniquely as $L = \Psi \circ P_V$, where $V$ is the $k$-dimensional orthogonal complement of $\Ker L$, $P_V$ is the orthogonal projection onto $V$ and $\Psi \colon V \to \R^k$ is a linear isomorphism depending continuously on $L$. It is easy to see that under this identification, full Lebesgue-measure sets in $\Lin(\R^N, \R^k)$ correspond to full $\gamma_{k,N}$-measure sets in $\Gr(k,N)$. Moreover, the injectivity of $L$ on a set $Y \subset \R^N$ is equivalent to the injectivity of $P_V$ on $Y$. Similarly, an equivalence holds between the continuity or pointwise H\"older continuity of the inverses of $L|_Y$ and $P_V|_Y$. 
\end{rem}

Extending the notation from the previous section, for a compactly supported finite measure $\mu$ in $\R^N$ and a linear space $V \in \Gr(k,N)$, we denote by $\{ \mu^V_a : a \in V \}$ the system of conditional measures of $\mu$ with respect to $P_{V}$. The measures $\mu^V_a$ are concentrated on $P_V^{-1}(a) = V^\perp + a$. By Lemma~\ref{lem:as injectivity Diracs} and Remark~\ref{rem:linear to projection}, Theorem~\ref{thm: hdim injectivity} implies the following corollary.

\begin{cor}\label{cor:Dirac slices}
	Let $\mu$ be a compactly supported finite Borel measure in $\R^N$ with $\hdim \mu < k$. Then for $\gamma_{k, N}$-almost every $V \in \Gr(k, N)$ and $W = V^\perp$, the sliced measure $\mu^V_{a}$ is a point mass for $P_{V} \mu$-almost every $a \in V$.
\end{cor}

J\"{a}rvenp\"{a}\"{a} and Mattila \cite[Theorem~3.3]{JMSections} proved a general `slicing' theorem for measures\footnote{J\"{a}rvenp\"{a}\"{a} and Mattila define `sliced measures' in \cite{JMSections} as weak-$^*$ limits of $\delta^{-k} \mu|_{P_V^{-1}(B(a, \delta))}$ as $\delta \searrow 0$, which exist for $\mH^k$-almost every $a \in V$. This is a different notion than conditional measures defined in Definition~\ref{def:conditional measures}, which satisfy $\mu^V_a = \lim \limits_{\delta \to 0} \left(\mu(P_V^{-1}(B(a, \delta)))\right)^{-1} \mu|_{P_V^{-1}(B(a, \delta))}$ (see \cite{SimmonsRohlin}). However, if $\lhdim \mu > k$ (which is the only non-trivial case in Theorem~\ref{thm:sections}), then $P_V \mu$ is absolutely continuous with respect to $\mH^k$ for $\gamma_{k,N}$-almost every $V \in \Gr(k,N)$ by Theorem~\ref{thm:marstrand}, so the two definitions are equal up to a constant for almost every $V \in \Gr(k,N)$ and $P_V \mu$-almost every $a \in V$. Moreover, Theorem~\ref{thm:sections} in \cite[Theorem~3.3]{JMSections} is valid for $\mH^k$-almost every $a \in V$ with $\mu^V_a(\R^N) > 0$, which is equivalent to the statement for $P_V \mu$-almost every $a \in V$ (by the same reason).}, which in our notation reads as follows.

\begin{thm}[{\bf Slicing Theorem}{}]\label{thm:sections}
Let $\mu$ be a compactly supported finite Borel measure in $\R^N$. Then for $\gamma_{k,N}$-almost every $V \in \Gr(k,N)$,
\[ \lhdim \mu^V_{a} \geq \lhdim \mu - k \quad \text{ for $P_V \mu$-almost every } a\in V.  \]
\end{thm}

Using the above result we can give a proof of Proposition~\ref{prop: dim bound}.

\begin{proof}[{Proof of Proposition~\rm\ref{prop: dim bound}}]
We claim that under the assumptions of the proposition, there exists a compact set 
$X \subset \R^N$ such that $\mu(X) > 0$ and $\nu = \mu|_X$ satisfies $\lhdim \nu > k$.
Indeed, recall (see e.g.~\cite[Proposition~10.3]{FalconerTechniques}) that
\[\hdim \mu = \underset{x \sim \mu}{\mathrm{ess sup}}\ \liminf \limits_{r \to 0} \frac{\log \mu(B(x,r))}{\log r}. \]
Now choose $X$ to be a compact subset of positive $\mu$-measure of the set
\[ \Big\{ x \in \R^N : \liminf \limits_{r \to 0} \frac{\log \mu(B(x,r))}{\log r} \geq s \Big\}\]
for fixed $s > 0$ with $k < s <\hdim \mu$. Applying Frostman's lemma (see e.g.~\cite[Theorem~8.6.3]{PU10}) we obtain $\lhdim \nu \geq s > k$.  By Theorem~\ref{thm:sections}, for $\gamma_{k, N}$-almost every $V \in \Gr(k, N)$,
 \[ \lhdim \nu^V_{a} \geq \lhdim \nu - k > 0 \text{ for } P_V \mu\text{-almost every } a \in V. \]
As Dirac's delta has dimension zero, this means that for $\gamma_{k,N}$-almost every $V \in \Gr(k,N)$, almost every conditional measure of $\nu$ with respect to $P_V$ is not a Dirac's delta, hence by Lemma~\ref{lem:as injectivity Diracs} there cannot exist a set of full $\nu$-measure on which $P_V$ is injective. As $\nu$ is absolutely continuous with respect to $\mu$, the same is true for $\mu$. By Remark~\ref{rem:linear to projection}, this holds also for almost every linear map $L \colon \R^N \to \R^k$.
\end{proof}

\section{Proof of Theorem~\ref{thm:inverses main}} \label{sec:proof_main}

Let $E = E(N,k)$ denote the set of linear maps $L \colon \R^N \to \R^k$ of the form
\[ L x = \big(\langle l_1, x \rangle , \ldots, \langle l_k, x \rangle\big),\]
where $l_1, \ldots, l_k \in \R^N$ satisfy $\|l_1\|, \ldots, \|l_k\| \leq 1$. As $E$ may be identified with $\left(B_N(0,1)\right)^k$, we will denote by $\Leb$ the normalized $k$-fold product of Lebesgue measures on $B_N(0,1)$, considered as a probability measure on $E$. Note that it is enough to prove the assertion of  Theorem~\ref{thm:inverses main} for Lebesgue-almost every $L \in E$, as then a rescaling gives the result for almost every linear mapping $L\colon\R^N \rightarrow \R^k$. By the Cauchy--Schwarz inequality, for all $x\in \R^N$,
\begin{equation}\label{eq:E}
	\|Lx\| \leq \sqrt{N} \,\|x\|.
\end{equation}

The following lemma is the key technical ingredient of the proof.

\begin{lem}\label{lem: key_ineq_linear}
	For every $x \in \R^N \setminus \{ 0 \}$, $z \in \R^k$ and $\eps>0$,
	\[
	\Leb (\{ L \in E : \|Lx + z \| \leq \eps \}) \leq C\frac{\eps^k}{\|x\|^k}, \]
	where $C > 0$ depends only on $N$ and $k$.
\end{lem}

For the proof see \cite[Lemma~4.1]{Rob11}. We will prove each of the assertions of Theorem~\ref{thm:inverses main} separately.

\subsection{Proof of assertion~\ref{it:hdim main} of Theorem~\ref{thm:inverses main}}
We actually prove the following, slightly stronger version of the statement.

\begin{thm}\label{thm:continuity extended}
Let $\mu$ be a finite Borel measure in $\R^N$ with a compact support $X$ satisfying $\mH^{k}(X) = 0$ and let $\phi\colon X \to \R^k$ be a Lipschitz map. Then for almost every linear map $L\colon \R^N \to \R^k$ there exists a Borel set $X_L \subset X$ of full $\mu$-measure such that for every $x \in X_L$ and every $\eps > 0$, there exists $\delta > 0$ for which the map $\phi_L = \phi + L$ satisfies
\begin{equation}\label{eq:continuous inverse}
\text{ for every } y \in X, \text{ if }	\|\phi_L(x) - \phi_L(y)\| \leq \delta, \text{ then } \| x - y \| \leq \eps.
\end{equation}
\end{thm}

\begin{proof}
The first part of the argument is obtained directly from the proof of \cite[Theorem~3.1]{BGS20}. We include the arguments for the convenience of the reader. First, we will prove that for every $x \in X$ we have
\begin{equation}\label{eq:almost every injectivity Leb} \Leb\big( \big\{ L\in E : \underset{y \in X\setminus \{x\}}{\exists}\  \phi_L(x) = \phi_L(y)  \big\} \big) = 0.
\end{equation}
Note that the above set, as well as all similar sets we consider in this section, are Borel measurable as a consequence of standard considerations (see \cite[Lemma~2.4]{BGS20}). As $\phi\colon X \to \R^k$ is a Lipschitz map, there exists $H>0$ so that for all $x,y\in X$,
\begin{equation}\label{eq:H2}
	\|\phi(x)-\phi(y)\| \leq H \,\|x-y\|.
\end{equation} 
Fix $x\in X$, $\eps>0$ and let
\[
K_{n} = \Big\{ y \in X :  \|x-y\| \geq\frac{1}{n} \Big\}
\]
for $n \in\N$. Define
\[
B_n = \big\{ L \in E : \underset{y \in K_n}{\exists}\ \phi_L (x) = \phi_L (y) \big\}.
\]
and note that for \eqref{eq:almost every injectivity Leb} it suffices to prove $\Leb(B_n) = 0$ for each $n$. As $\hdim K_n\leq \hdim X < k$, there exists a collection of balls $B_N(y_i, \eps_i)$, $i \in \N$, for some $y_i \in K_{n}$ and $\eps_i > 0$, such that
\begin{equation*} K_n \subset \bigcup \limits_{i \in \N} B_N(y_i, \eps_i)\quad \text{and} \quad  \sum \limits_{i = 1}^\infty \eps_i^{k} \leq \eps.
\end{equation*}
Let $L \in B_n$. Then there is $y \in K_n$ such that $\phi_L (x) = \phi_L (y)$. Clearly, $y \in B_N (y_i, \eps_i)$ for some $i \in \N$. We calculate
\begin{align*}
	\| \phi_L(x)-\phi_L(y_i)  \| &\leq \| \phi_L(x) - \phi_L(y)\|+\| \phi_L(y) - \phi_L(y_i)\|\\
	&= \| \phi_L(y) - \phi_L(y_i)\|\\
	& \leq \|\phi(y_i) - \phi(y)\| + \| L(y_i-y)\|\\
	&\leq H \|y_i-y\| + \sqrt{N} \|y_i - y\|\\
	&\leq M \eps_i
\end{align*}
for  $M=H+ \sqrt{N}$, by \eqref{eq:E} and \eqref{eq:H2}. This shows 
\[ B_n \subset \bigcup \limits_{i \in \N} \{ L \in E : \| \phi_L(x)-\phi_L(y_i) \| \leq M\eps_i \}. \]
Thus, using Lemma~\ref{lem: key_ineq_linear} and the fact $\| x-y_i\|\geq \frac{1}{n}$, we obtain
\begin{align*}
	\Leb(B_n) &\leq \sum \limits_{i = 1}^\infty \Leb(\{ L \in E : \| L(x-y_i)+\phi(x)-\phi(y_i) \| \leq M\eps_i \})\\
	&\leq \frac{C M^k}{1/n^k}\sum \limits_{i = 1}^\infty \eps_i^{ k} \leq  C M^k n^k\eps.
\end{align*}
As  $\eps > 0$ is arbitrary, we obtain $\Leb(B_n) = 0$, and thus \eqref{eq:almost every injectivity Leb} is established.
Combining \eqref{eq:almost every injectivity Leb} with Fubini's theorem (see e.g.~\cite[Theorem~8.8]{R87}), we obtain 
\begin{equation}\label{eq:almost every injectivity}\mu \big( \big\{ x \in X : \underset{y \in X \setminus \{ x \}}{\exists}\  \phi_{L}(x) = \phi_L(y)  \big\} \big) = 0
\end{equation}
for Lebesgue-almost every $L\in E$. Hence, the set
\[ X_L = X \setminus \big\{ x \in X : \underset{y \in X \setminus \{ x \}}{\exists} \phi_{L}(x) = \phi_L(y)  \big\} \]
is a full $\mu$-measure set on which $\phi_L$ is injective (which proves Theorem~\ref{thm: hdim injectivity}). To obtain additionally the continuity of $\phi_L^{-1}$ on $X_L$, fix $L \in E$ satisfying \eqref{eq:almost every injectivity}. We claim that every $x \in X_L$ satisfies  \eqref{eq:continuous inverse}. If not, then there exists $\eps > 0$ such that for every $n \geq 1$ there exists $y_n \in X$ satisfying
\[
\| \phi_L(x) - \phi_L(y_n)\| \leq \frac{1}{n} \quad \text{and} \quad \| x-y_n\|>  \eps.
\]
As $X$ is compact, there is a converging subsequence $y_{n_k} \to y$ for some $y=y(L, x) \in X$. By the continuity of $\phi_L$, we have $\phi_L(x) = \phi_L(y)$ and $  \| x-y\|\geq  \eps$, in particular $x\neq y$, contradicting $x \in X_L$.
\end{proof}

\subsection{Proof of assertion~\ref{it:udim main} of Theorem~\ref{thm:inverses main}}
The proof combines the techniques of \cite[Theorem~3.1]{HK99} and \cite[Theorem~3.1]{BGS20}.

\begin{lem}\label{lem: eps delta Leb}
Let $X$ be a compact subset of $\R^N$ with $\udim X<k$. Fix $\theta \in \left(0, k - \udim X\right)$ and a Lipschitz map $\phi\colon X \to \R^k$. Then there exists a constant $D>0$ such that 
\[ \Leb\big( \big\{ L \in E : \underset{y \in X}{\exists}\ \|\phi_L(x) - \phi_L(y)\| \leq \eps  \text{ and }  \|x - y\| \geq \delta  \big\} \big) \leq D \delta^{-k}\eps^{k - \udim X - \theta} \]
for every $x \in X$ and $0 < 2\eps \leq \delta$, where $\phi_L = \phi+L$.
\end{lem}

\begin{proof} By the definition of $\udim$, there exists a constant $D=D(X, \theta)$ such that for every $\eps>0$ there exists a cover
\begin{equation}\label{eq:eps cover} X \subset \bigcup \limits_{i=1}^{N_{\eps}} B(y_i, \eps)\ \text{ with }\ N_{\eps} \leq D \eps^{-(d + \theta)}.
\end{equation}
Consider $x,y \in X$ and $L \in E$ such that $\|\phi_L(x) - \phi_L(y)\| \leq \eps$ and  $\|x - y\| \geq \delta$. Let $y_i$ be such that $y \in B(y_i, \eps)$. Then, recalling that we assume $2\eps \leq \delta$,
\[\begin{split} \|L(x - y_i) - \left(\phi(y_i) - \phi(x)\right)\| &= \|\phi_L(x) - \phi_L(y_i)\| \leq \|\phi_L(x) - \phi_L(y)\| + \|\phi_L(y) - \phi_L(y_i)\| \\
& \leq \eps + \big(\Lip(\phi) + \sup \limits_{L \in E}\| L \|\big)\eps \leq M\eps,
\end{split}\]
where $M = 1 + \Lip(\phi) + \sup \limits_{L \in E}\| L \| < \infty$. Moreover,
\[ \|x - y_i\| \geq \|x - y\| - \|y - y_i\| \geq \delta - \eps \geq \frac \delta 2. \]
Therefore,
\begin{align*} &\big\{ L \in E : \underset{y \in X}{\exists}\ \|\phi_L(x) - \phi_L(y)\| \leq \eps  \text{ and }  \|x - y\| \geq \delta  \big\} \\
&\subset \bigcup \limits_{i=1}^{N_\eps} \left\{ L \in E : \|L(x - y_i) -\left(\phi(y_i) - \phi(x)\right) \| \leq M\eps  \text{ and }  \|x - y_i\| \geq \delta/2  \right\}.
\end{align*}
Hence, by Lemma~\ref{lem: key_ineq_linear} and \eqref{eq:eps cover},
\[
\begin{split}
\Leb & \big( \big\{ L \in E : \underset{y \in X}{\exists}\ \|\phi_L(x) - \phi_L(y)\| \leq \eps  \text{ and }  \|x - y\|
 \geq \delta  \big\} \big) \\
  & \leq C2^k N_\eps \delta^{-k} M^k\eps^k \leq C2^k D M^k \delta^{-k}\eps^{k - d - \theta}.
\end{split}
\]

\end{proof}

\begin{proof}[{Proof of Theorem~\rm\ref{thm:inverses main}\ref{it:udim main}}]
Set $d = \udim(\supp\mu)$. Fix $\alpha \in \left(0, 1 -  \frac{d}{k}\right)$ and let $\theta \in (0, k - d)$ be such that $\alpha < 1 - \frac{d+ \theta}{k}$. Let $H = \sup \limits_{L \in E} \diam \left( \phi_L(X)\right)$. For a fixed $x \in X$, by \eqref{eq:almost every injectivity Leb}  (which we can apply as $\hdim \mu \leq d < k$) and Lemma~\ref{lem: eps delta Leb},
\[\begin{split}
\Leb &  \big( \big\{ L \in E : \underset{M>0}{\forall}\ \underset{y \in X}{\exists}\ \|x - y\| > M\| \phi_L(x) - \phi_L(y) \|^\alpha  \big\} \big) \\
& = \lim \limits_{M \to \infty} \Leb \big( \big\{ L \in E : \underset{y \in X}{\exists}\ \|x - y\| > M\| \phi_L(x) - \phi_L(y) \|^\alpha  \big\} \big) \\
& =   \lim \limits_{M \to \infty} \sum \limits_{m=0}^\infty \Leb\big( \big\{ L \in E : \underset{y \in X}{\exists}\ 2^{-(m+1)}H  <\|\phi_L(x) - \phi_L(y)\| \leq 2^{-m}H  \\
& \hspace{150pt}  \text{ and }  \|x - y\| > M\|\phi_L(x) - \phi_L(y)\|^\alpha  \big\} \big)  \\
& \hspace{20pt} + \Leb \big( \big\{ L \in E : \underset{y \in X}{\exists}\ \phi_L(x) = \phi_L(y) \text{ and } \|x - y \| > 0 \big\} \big) \\
& \overset{\eqref{eq:almost every injectivity Leb}}{\leq}  \lim \limits_{M \to \infty} \sum \limits_{m=0}^\infty \Leb\big( \big\{ L \in E : \underset{y \in X}{\exists}\ \|\phi_L(x) - \phi_L(y)\| \leq 2^{-m}H  \text{ and }  \|x - y\| > MH^\alpha2^{-\alpha(m+1)}  \big\} \big)  \\
& \overset{\mathclap{\text{Lemma }\ref{lem: eps delta Leb}}}{\leq}\ \ \ \lim \limits_{M \to \infty} \sum \limits_{m=0}^\infty D M^{-k}H^{-\alpha k}2^{\alpha k (m+1)} 2^{-m(k - d - \theta)}H^{k - d - \theta} \\
& = \lim \limits_{M \to \infty} D{M^{-k}}H^{k(1-\alpha) - d - \theta}2^{\alpha k} \sum \limits_{m=0}^\infty 2^{m(\alpha k - k + d + \theta)} = 0,
\end{split}
\]
as the series $\sum \limits_{m=0}^\infty 2^{m(\alpha k - k + d + \theta)}$ converges since we assume $\alpha < 1 - \frac{d+ \theta}{k}$.
This proves that for every $x \in X$, the condition 
\begin{equation}\label{eq:pointwise holder inverse} \|x - y\| \leq M\|\phi_L (x) - \phi_L (y)\|^\alpha \quad \text {for some } M=M(x,L) \text{ and every } y \in X
\end{equation}
is satisfied for almost every $L \in E$. Therefore, by Fubini's theorem, for almost every $L \in E$, the condition \eqref{eq:pointwise holder inverse} holds for $\mu$-almost every $x\in X$. Finally, note that by taking a countable intersection of full Lebesgue measure sets, we can assume that for almost every $L \in E$, the condition \eqref{eq:pointwise holder inverse} holds for every $\alpha < 1 - \frac{d+ \theta}{k}$.
\end{proof}

\subsection{Proof of assertion~\ref{it:adim main} of Theorem~\ref{thm:inverses main}}
Again, we prove a stronger result.

\begin{thm}\label{thm:assouad full}
	Let $R>0$ and $\eta>1$. Let $\mu$  be a probability measure in $\R^N$ with a compact support $X$ satisfying $|X| \leq R$ and $\dim_A X < k$. Let $\phi\colon X \to \R^k$ be a Lipschitz map and fix $\theta\in (0,k-\adim X)$. Then for almost every linear mapping
	$L\colon \R^N \rightarrow \R^k$ there exists a Borel set $X_L \subset X$ of full $\mu$-measure such that for each point $x \in X_L$ there exists $C> 0$ for which the map $\phi_L = \phi + L$ satisfies
	\begin{equation}\label{eq:pointwise log-Lip inverse}
		\|\phi_L (x) - \phi_L (y)\| \geq C f(\|x-y\|) \quad \text{for every } y \in X,
	\end{equation}
	where
	$$f(x)= \frac{x}{(\log_2(2R/x))^{\eta/\theta}}.$$
	
\end{thm}

\begin{proof}
Once more, by Fubini's theorem it is enough to prove that for every $x \in X$, the condition \eqref{eq:pointwise log-Lip inverse} holds for almost every $L\in E$. The rest of the proof is a combination of  the methods set forth in \cite[Theorem~5.2]{olson2002bouligand} and \cite[Theorem~3.1]{BGS20}. As $\phi\colon X \to \R^k$ is a Lipschitz map, there exists $H>0$ so that for all $x,y\in X$,
\begin{equation}\label{eq:H}
		\|\phi(x)-\phi(y)\| \leq H \,\|x-y\|.
	\end{equation} 
	Fix $x\in X$. Define $r_n=\frac{R}{2^n}$ and $\rho_n=f(r_{n-1})>0$. Note that for $0<x\leq R$, we have $\frac{f(x)}{x}\leq 1$ as $\log_2 \frac{2R}{x}\geq 1$. Moreover, a simple calculation shows that the function $f$ is monotone increasing on $(0, R]$. For $n\geq 1$, define
	$$Z_n=\{y\in X : \,r_n\leq \|y-x\|\leq r_{n-1}\}.$$
		By the definition of the Assouad dimension, there exists $K>0$ such that for every $0<s<r_{n-1}$ and a  ball $B$ of radius $r_{n-1}$, the set 
	$X\cap B$ may be covered by $K(\frac{r_{n-1}}{s})^{k-\theta}$ balls of radius $s$. Let $c>2$ satisfy $c^\theta >K$. We conclude that the set $Z_n$, which is contained in a ball $B$ of radius $r_{n-1}$ around $x$, may be covered by at most $\ell_{n,i}\leq K c^{i(k-\theta)}(\frac{r_{n-1}}{\rho_n})^{k-\theta}$ balls $\{B(a_{n,i,j}, \frac{\rho_n}{c^i})\}_{j=1}^{\ell_{n,i}}$  of radius $\frac{\rho_n}{c^i}$ (with centers in $Z_n$) for $i\geq 1$ (recall that $\rho_n\leq r_{n-1})$. Thus, $Z_n\subset \bigcup \limits_{j=1}^{\ell_{n,i}} B(a_{n,i,j}, \frac{\rho_n}{c^i})$. For $i\geq 2$, define
		$$U_i=\bigcup \limits_{n=1}^\infty \bigcup \limits_{j = 1}^{\ell_{n,i}} B\Big(a_{n,i,j}, \frac{2\rho_n}{c^i}\Big).$$
		Every center $a_{n,i,j}$  satisfies $\|x-a_{n,i,j}\|\geq r_n$, so the ball centred at $a_{n,i,j}$ of radius $\frac{2\rho_n}{c^i}<\frac{\rho_n}{2}\leq \frac{r_{n-1}}{2}\leq r_n$ does not contain $x$. Thus, for $i\geq 2$, 
	$$X\setminus \{x\} \subset U_i.$$
	In order to establish the condition \eqref{eq:pointwise log-Lip inverse} for a fixed $x \in X$ and Lebesgue-almost every $L\in E$, it is enough to show
		$$\lim_{i\rightarrow \infty} \Leb \big\{ L \in E : \underset{y \in X}{\exists}\,\,\underset{n\geq 1}{\exists}\,\, \|\phi_L(x) - \phi_L(y)\| < \frac{\rho_n}{c^i}  \text{ and }  \|x - y\| \geq r_n  \big\}=0$$
		Indeed, this implies that for almost every $L \in E$ there exists $i=i(L)\geq 2$, such that for every $y\in X$, $\|x - y\| \geq r_n$ implies $\|\phi_L(x) - \phi_L(y)\| \geq \frac{\rho_n}{c^i}$. As every $y\in X\setminus \{x\}$ is contained in some $Z_n$ (recall that $|X| \leq R$), this implies that for every $y\in X\setminus \{x\}$ (using monotonicity of $f$ on $(0,R]$), we obtain
		$$\|\phi_L(x) - \phi_L(y)\|\geq \frac{\rho_n}{c^i}=\frac{f(r_{n-1})}{c^i}\geq \frac{1}{c^i}f(\|x - y\|).$$
	
Let
\[
A_i=\big\{ L \in E : \underset{y \in X}{\exists}\,\,\underset{n\geq 1}{\exists}\,\, \|\phi_L(x) - \phi_L(y)\| < \frac{\rho_n}{c^i}  \text{ and }  \|x - y\| \geq r_n  \big\}.
\]
Clearly, $A_i$ is a Borel set. For $L\in A_i$ one may find $y\in X$ and $n\geq 1$ such that $\|\phi_L(x) - \phi_L(y)\| < \frac{\rho_n}{c^i}$ and $\|x - y\| \geq r_n$. Consequently, $y\in \bigcup_{m=1}^{n}Z_m$. Therefore, one may find a center $a_{m,i,j}$ such that $y\in B(a_{m,i,j}, \frac{\rho_m}{c^i})$. Note that
		$$\|\phi_L(x) - \phi_L (a_{m,i,j})\|\leq \|\phi_L(x) - \phi_L(y)\|+\|\phi_L(y) - \phi_L (a_{m,i,j})\|.$$
		By \eqref{eq:E} and \eqref{eq:H},
		$$\|\phi_L(y) - \phi_L (a_{m,i,j})\|\leq \|Ly - La_{m,i,j}\|+\|\phi(y) - \phi(a_{m,i,j})\|\leq \sqrt{N}\frac{\rho_m}{c^i}+H\frac{\rho_m}{c^i}. $$
	Thus,
	$$
	\|\phi_L(x) - \phi_L (a_{m,i,j})\|\leq \frac{\rho_n}{c^i}+(\sqrt{N}+H)\frac{\rho_m}{c^i}\leq Q\frac{\rho_m}{c^i},$$
	where $Q=1+\sqrt{N}+H$. We conclude
	$$A_i\subset \bigcup_{m=1}^{\infty}\bigcup_{j=1}^{\ell_{m,i}}A_{m,i,j}$$
	for
	\[A_{m,i,j}=\left\{ L \in E : \,\, \|\phi_L(x) - \phi(a_{m,i,j}) \|\leq  Q\frac{\rho_m}{c^i} \right\}.
	\]
	By Lemma~\ref{lem: key_ineq_linear} (recall that $a_{m,i,j} \in Z_m$),
	$$\Leb (A_{m,i,j})\leq C\frac{(Q\frac{\rho_m}{c^i})^k}{r_m^k}.$$
	Thus,
	$$\Leb (A_i)\leq\sum_{m=1}^{\infty}\sum_{j=1}^{\ell_{m,i}}\Leb (A_{m,i,j})\leq \sum_{m=1}^{\infty} K c^{i(k-\theta)}\Big(\frac{r_{m-1}}{\rho_m}\Big)^{k-\theta}C\frac{(Q\frac{\rho_m}{c^i})^k}{r_m^k}\leq \frac{KC(2Q)^k} {c^{i\theta}}\sum_{m=1}^{\infty}\left(\frac{f(r_{m-1})}{r_{m-1}}\right)^\theta.$$
	We notice
	$$\left(\frac{f(r_{m-1})}{r_{m-1}}\right)^\theta= \frac{1}{(\log_2(2^m))^{\eta}}=\frac{1}{m^\eta}.$$
	Thus,
	$$\Leb (A_i)\leq\frac{KC(2Q)^k} {c^{i\theta}}\sum_{m=1}^{\infty}\frac{1}{m^\eta}.$$
	As $\sum \limits_{m=1}^{\infty}\frac{1}{m^\eta} < \infty$, this implies $\lim \limits_{i\rightarrow \infty} \Leb (A_i)=0$, which ends the proof.
\end{proof}

\section{Measures with all almost-surely injective projections} \label{sec:every_direction}

\subsection{Proof of Theorem~\ref{thm:every direction a.s. injectivity}}
The measure $\mu$ will be defined in two steps. First, we define a measure $\nu$ on the unit interval by randomizing digits in dyadic expansions and then push $\nu$ to the graph of the function $x \mapsto x^2$. For the first step, it is convenient to work in the symbolic space $\{ 0,1 \}^{\N}$.

Partition $\N = \{1, 2, \ldots \}$ into blocks $B_n = \{2n-1, 2n \}$, $n \geq 1$ and denote $L_n = 2n-1$, $R_n = 2n$, so that the block $B_n$ consists of the left bit $L_n$ and the right bit $R_n$. For $\om \in \{0,1\}^{\N}$ we use the notation $\om = (\om_1, \om_2, \ldots)$. Set
\[ \Sigma = \{ \om \in \{0,1\}^\N : \om_{L_n} = 0 \text{ for every } n \geq 1 \}. \]
Fix $p \in (0,1/2)$. Define a probability measure $\mathbf{p}$ on $\{0,1\}^2$ as
\[ \mathbf{p} = p\delta_{(0,1)} + (1-p)\delta_{(0,0)}\]
and let $\PP$ be a probability measure on $\{0,1\}^{\N}$ given by 
\[ \PP = \mathbf{p}^{\otimes \N} \]
(we use below the identification $\{0,1\}^\N = \left(\{0,1\}^2\right)^\N$).
Clearly, $\PP(\Sigma) = 1$. Now we transport $\PP$ to the unit interval by setting
\[ \nu = \Pi\PP, \qquad X = \Pi(\Sigma), \]
where $\Pi$ is given by 
\begin{equation*}\Pi(\om) = \sum \limits_{j=1}^{\infty} \om_j 2^{-j}.
\end{equation*}
The last step is to lift $\nu$ to the graph of a non-linear function. Let $f\colon [0,1] \to [0,1]$ be given by $f(x) = x^2$ and $\Psi\colon [0,1] \to [0,1]^2$ by $\Psi(x) = (x, f(x))$. Finally, we set
\[ \mu = \Psi \nu \]
and claim that the measure $\mu$ satisfies the properties from Theorem~\ref{thm:every direction a.s. injectivity}. First, note that $\hdim \nu = \frac{- p\log p - (1-p)\log(1-p) }{\log 4}$ (see \cite[Example~1.5.6]{BP17}). As $\Psi$ is bi-Lipschitz, we have $\hdim \mu = \hdim \nu > 0$. For the injectivity part of Theorem~\ref{thm:every direction a.s. injectivity}, fix a non-zero linear map $L\colon \R^2 \to \R,\ L(x,y) = \alpha x + \beta y$. As $f$ is a bijection on $[0,1]$, we can assume that $\alpha \neq 0$ and $\beta \neq 0$, as otherwise $L$ is injective on the whole graph $\Psi([0,1])$ and the claim of Theorem~\ref{thm:every direction a.s. injectivity} follows trivially. Note that for $x,y \in [0,1]$ with $x \neq y$ we have
\begin{equation}\label{eq: Lx = Ly}
	L(x, f(x)) = L(y, f(y)) \iff \frac{f(y) - f(x)}{x - y} = \frac{\alpha}{\beta}   \iff x + y = - \frac{\alpha}{\beta}.
\end{equation}
Therefore, in order to show the injectivity of $L$, we need to study solutions of the equation $x+y = z$ for fixed $z$ and $x,y$ taken from $X$. We will do so in terms of the dyadic expansions.

Note that  every $z \in [0,1)$ has a unique dyadic expansion $z = \sum_{j=1}^\infty z_j 2^{-j}$ such that the sequence $(z_j)_{j=1}^\infty \in \{0,1\}^\N$ is not eventually equal to $1$ (we say that the dyadic expansions do not terminate with $1$'s). Moreover, the only points $z \in [0,1)$ which have a non-unique dyadic expansion are the dyadic rationals from $(0,1)$. For them, there are exactly two expansions, one terminating with $1$'s and one terminating with $0$'s. As $\Sigma$ does not contain sequences terminating with $1$'s, we see that $\Pi$ is injective on $\Sigma$. Therefore, for $x,y \in X$ we will write $x = \sum \limits_{j=1}^\infty x_j 2^{-j},\ y = \sum \limits_{j=1}^\infty y_j 2^{-j}$ for its unique dyadic expansion which does not terminate with $1$'s, so that $\Pi((x_1, x_2, \ldots)) = x,\ \Pi((y_1, y_2, \ldots)) = y$ and $(x_1, x_2, \ldots), (y_1, y_2, \ldots) \in \Sigma$. First, we need a technical lemma.

\begin{lem}\label{lem: x+y = z}
	Let $x,y,z \in [0,1)$ have dyadic expansions  
	$$x = \sum_{j=1}^{\infty} x_j 2^{-j},\quad y = \sum_{j=1}^{\infty} y_j 2^{-j},\quad z = \sum_{j=1}^{\infty} z_j 2^{-j}$$ 
	which do not terminate with $1$'s. Assume $x + y = z$. Then for every $k \in \N$, the condition $x_k = y_k = 0$ implies $\sum \limits_{j<k}z_j 2^{-j} = \sum \limits_{j<k} (x_j + y_j)2^{-j}$. Consequently, for every $k, m \in \N$ with $k < m$ such that $x_k = y_k = x_m = y_m=0$ we have $$\sum \limits_{k \leq j< m}z_j 2^{-j} = \sum \limits_{k < j < m} (x_j + y_j)2^{-j}.$$
\end{lem}

\begin{proof}
	We begin by proving the first assertion. Assume $x_k = y_k = 0$. Then
	\[ \sum \limits_{j<k}z_j 2^{-j} = a2^{-(k-1)}, \qquad \sum \limits_{j<k} (x_j + y_j)2^{-j} = b2^{-(k-1)}\]
	for some $a,b \in \N\cup \{ 0 \}$. As $x + y = z$, we have
	\begin{equation}\label{eq: x+y=z} a2^{-{(k-1)}} + \sum \limits_{j \geq k}z_j 2^{-j}   = b2^{-(k-1)} + \sum \limits_{j \geq k} (x_j + y_j)2^{-j}.
	\end{equation}
	Since $x_j, y_j, z_j$ are not eventually equal to $1$ and $x_k = y_k = 0$, we obtain
	\[ 0 \leq \sum \limits_{j \geq k}z_j 2^{-j} < 2^{-(k-1)}, \qquad 0 \leq \sum \limits_{j \geq k} (x_j + y_j)2^{-j} < 2^{-(k-1)}. \]
	Combining this with \eqref{eq: x+y=z} yields
	\[ |a - b|2^{-(k-1)} = \Big| \sum \limits_{j \geq k}z_j 2^{-j} - \sum \limits_{j \geq k} (x_j + y_j)2^{-j}  \Big| < 2^{-(k-1)}, \]
	so $a = b$ and, consequently, $\sum_{j<k}z_j 2^{-j} = \sum_{j<k} (x_j + y_j)2^{-j}$.
	
	Now, if $x_k = y_k = x_m = y_m=0$, then by the first assertion,
	\[
	\begin{split} \sum \limits_{k \leq j< m}z_j 2^{-j} & = \sum \limits_{ j< m}z_j 2^{-j} - \sum \limits_{ j< k}z_j 2^{-j} = \sum \limits_{ j< m}(x_j + y_j) 2^{-j} - \sum \limits_{ j< k}(x_j + y_j) 2^{-j} \\
		& =  \sum \limits_{k \leq j< m}(x_j + y_j) 2^{-j} = \sum \limits_{k < j< m}(x_j + y_j) 2^{-j}.
	\end{split}
	\]
\end{proof}
The above lemma and the structure of $\Sigma$ provide the following consequence.
\begin{lem}\label{lem:z forces x,y}
	Let $z = x + y$ for some $x,y \in X$, such that $z$ has a dyadic expansion $z = \sum_{j=1}^{\infty} z_j 2^{-j}$ which does not terminate with $1$'s. Then, denoting $z|_{B_n} = (z_{L_n}, z_{R_n})$, for every $n \geq 1$ we have the following.
	\begin{enumerate}[$($i$)$]
		\item\label{it:z Bn} $z|_{B_n} \in \{ (0,0), (0,1), (1,0)\}$.
		\item\label{it:z Bn 00} If $z|_{B_n} = (0,0)$, then $x_{R_n} = y_{R_n} = 0$.
		\item\label{it:z Bn 10} If $z|_{B_n} = (1,0)$, then $x_{R_n} = y_{R_n} = 1$.
		\item\label{it:z Bn 01} If $z|_{B_n} = (0,1)$, then exactly one of bits $x_{R_n}, y_{R_n}$ is equal to $1$.
	\end{enumerate}
\end{lem}

\begin{proof}
	As $x_{L_n} = y_{L_n} =  0$ for $n \geq 1$, we see from Lemma~\ref{lem: x+y = z} that
	\[ \sum_{j \in B_n} z_j 2^{-j} =  (x_{R_n} + y_{R_n}) 2^{-R_n}.\]
	The assertions of the lemma follow from this equality by recalling that $B_n = \{ R_n -1, R_n\} = \{ L_n, R_n\}$.
\end{proof}

Theorem~\ref{thm:every direction a.s. injectivity} is an immediate consequence of the following proposition.

\begin{prop}\label{prop:new}
	Let $L\colon \R^2 \to \R$ be a linear map of the form $L(x,y) = \alpha x + \beta y$ with $\alpha \neq 0$ and $\beta \neq 0$. Then there exists a Borel set $\Sigma_L \subset \Sigma$ with $\PP(\Sigma_L) = 0$, such that if $x,y \in X$, $x\neq y$ have dyadic expansions $x = \sum_{j=1}^{\infty} x_j 2^{-j}$, $y = \sum_{j=1}^{\infty} y_j 2^{-j}$ which do not terminate with $1$'s and satisfy $L(x, f(x)) = L(y, f(y))$, then at least one of the sequences $(x_1, x_2, \ldots), (y_1, y_2, \ldots)$ belongs to $\Sigma_L$.
\end{prop}

Indeed, if we set $X_L = \Psi(X \setminus \Pi(\Sigma_L))$, then $X_L$ is Borel (by \cite[Theorem~5.1]{K95}), $\mu(X_L) = 1$ and $L$ is injective on $X_L$.

\begin{proof}[Proof of Proposition~\rm\ref{prop:new}]
	Assume that there exist $x,y \in X$ such that $L(x, f(x)) = L(y, f(y))$, as otherwise there is nothing to prove. Set $z = z(L) = - \frac{\alpha}{\beta}$. 	Note that $z \in [0,1)$ as $x_1 = y_1 = 0$ and the expansions of $x$ and $y$ do not terminate with $1$'s. Therefore, $z$ has a unique dyadic expansion $z = \sum_{j=1}^{\infty} z_j 2^{-j}$ which does not terminate with $1$'s.  We define the set $\Sigma_L$ in terms of the sequence $(z_j)_{j=1}^\infty$.
	
	First, assume that $z|_{B_n} \in \{(0,0), (1,0)\}$ for infinitely many $n \geq 1$. In this case we set
	\[ \Sigma_L = \big\{ \om \in \{0,1 \}^\N : \text{for each } n\geq 1,\  \om_{R_n} = 0 \text{ if } z|_{B_n} = (0,0) \text{ and } \om_{R_n} = 1 \text{ if } z|_{B_n} = (1,0) \big\}. \]
	Clearly, $\PP(\Sigma_L) = 0$. If $x,y \in X$, $x \neq y$ with dyadic expansions $x = \Pi(x_1, x_2, \ldots)$, $y = \Pi(y_1, y_2, \ldots)$ satisfy $L(x, f(x)) = L(y, f(y))$, then by \eqref{eq: Lx = Ly}, we have $x + y = z$, hence Lemma~\ref{lem:z forces x,y} implies $(x_1, x_2, \ldots) \in \Sigma_L$ and $(y_1, y_2, \ldots) \in \Sigma_L$.
	
	The remaining case is the one with $z|_{B_n} \in \{ (0,0), (1,0) \}$ only for finitely many $n \geq 1$. Then, by Lemma~\ref{lem:z forces x,y} we must have $z|_{B_n} = (0,1)$ for all $n$ large enough. Set
	\[ \Sigma_L = \Big\{ \om \in \{0,1 \}^\N : \limsup \limits_{N \to \infty} \frac{1}{N} \# \big\{ 1 \leq n \leq N : \om_{R_n} = 1 \big\} \geq \frac 1 2 \Big\}. \]
	As $p < 1/2$, the Birkhoff ergodic theorem gives $\PP(\Sigma_L) = 0$. Again, if $x,y \in X,\ x \neq y$ satisfy $L(x, f(x)) = L(y, f(y))$, then $x+y = z$, so by Lemma~\ref{lem:z forces x,y} and the assumption on $z$, we have
	\[
	\begin{split} 
		1 = &  \limsup \limits_{N \to \infty} \frac{1}{N} \# \left\{ 1 \leq n \leq N : z_{B_n} = (0,1) \right\}\\
		& \leq \limsup \limits_{N \to \infty} \frac{1}{N} \# \left\{ 1 \leq n \leq N : x_{R_n} = 1 \text{ or } y_{R_n} = 1 \right\} 	\\
		& \leq \limsup \limits_{N \to \infty} \frac{1}{N}  \# \left\{ 1 \leq n \leq N : x_{R_n} = 1 \right\} +  \limsup \limits_{N \to \infty} \frac{1}{N} \# \left\{ 1 \leq n \leq N : y_{R_n} = 1 \right\}.
	\end{split}\]
	Therefore, $(x_1, x_2, \ldots) \in \Sigma_L$ or $(y_1, y_2, \ldots) \in \Sigma_L$. This concludes the proof of the proposition.
\end{proof}

\subsection{Self-similar measures}

In this subsection we will prove Proposition~\ref{prop:self-sim}. We actually obtain a stronger statement. Let $\vphi_i\colon \R^N \to \R^N$, $i \in I$, be a finite collection of contractions and let $(p_i)_{i \in I}$ be a strictly positive probability vector. It is well-known (see \cite{HutchinsonFractals}) that there exists a unique Borel probability measure in $\R^N$ which is stationary for this system, i.e.~satisfies 
\[  \mu = \sum \limits_{i \in I} p_i \,\vphi_i \mu.\]
\begin{prop}\label{prop:ifs with translate}
Assume that the iterated function system given by $\{ \varphi_i : i\in I \}$ satisfies $\varphi_{i_2} = \varphi_{i_1} + t$ for $i_1, i_2 \in I$, $i_1 \ne i_2$ and some non-zero vector $t \in \R^N$. Let $\mu$ be the stationary measure for a probability vector $(p_i)_{i \in I}$.  Consider $V \in \Gr(k, N)$ contained in the orthogonal complement of $t$. Then $P_V$ is not injective on any set of full $\mu$-measure.
\end{prop}

\begin{proof} We can assume $i_1 = 1$, $i_2 = 2$.
Let $\nu = p_1 \,\vphi_1 \mu + p_2\, \vphi_2 \mu$. As $\nu$ is absolutely continuous with respect to $\mu$, it is enough to prove the statement for the measure $\nu$. Consider the system of conditional measures $\{ \nu_{a} : a \in V \}$ of $\nu$ with respect to the map $P_V$. Let $\nu^{(1)} = \vphi_1 \mu,\ \nu^{(2)} = \vphi_2 \mu$, so that 
\begin{equation}\label{eq:nu sum}
\nu = p_1 \nu^{(1)} + p_2 \nu^{(2)}
\end{equation}
and let $\{ \nu^{(i)}_{a} : a \in V \},\ i =1,2$ be the system of conditional measures of $\nu^{(i)}$ with respect to $P_V$. Let
\[
S_t\colon \R^N \to \R^N,\qquad S_t(x) = x +t.
\]
A crucial observation is that
\begin{equation}\label{eq:shifted slices} \nu^{(2)}_a = S_{t}\nu^{(1)}_a \quad \text{for almost every } a \in V \text{ with respect to the measure } P_V\nu^{(1)} = P_V\nu^{(2)}.
\end{equation}
To verify that, note first that as $\vphi_2(x) = \vphi_1(x) + t$ for all $x \in \R^N$, we have $\nu^{(2)}=S_t\nu^{(1)}$. As $V$ is contained in the orthogonal complement of $t$, we have $P_V \circ S_t = P_V$, so
\begin{equation}\label{eq:equal projections}
	P_V \nu^{(2)} = P_V S_t \nu^{(1)} = P_V \nu^{(1)}.
\end{equation}
Therefore, to obtain \eqref{eq:shifted slices}, it is enough to check that $\{ S_t\nu_a^{(1)} : a \in V \}$ is a system of conditional measures of $\nu^{(2)}$ with respect to $P_V$. As $\nu_a^{(1)}(P_V^{-1}(a)) = 1$, we have $$S_t\nu_a^{(1)} (P_V^{-1}(a)) = \nu_a^{(1)} (S_t^{-1}P_V^{-1}(a)) = \nu_a^{(1)} (P_V^{-1}(a)) = 1,$$ hence we see from \eqref{eq:equal projections} that the conditions~(i)--(ii) of Definition~\ref{def:conditional measures} are satisfied. For the condition~(iii), we use \eqref{eq:equal projections} once more to obtain
\[ \nu^{(2)} (A) = \nu^{(1)}(S_t^{-1}A) = \int_{V} \nu^{(1)}_a(S_t^{-1} A)dP_V\nu^{(1)}(a) = \int_{V} S_t\nu^{(1)}_a(A)dP_V\nu^{(2)}(a).  \]
This shows that indeed $\{ S_t\nu_a^{(1)} : a \in V \}$ is a system of conditional measures for $\nu^{(2)}$ with respect to $P_V$, hence \eqref{eq:shifted slices} is established by its uniqueness. By \eqref{eq:nu sum}  and \eqref{eq:equal projections} we have
\begin{equation}\label{eq:nu sum slices} \nu_{a} = p_1 \nu^{(1)}_a + p_2 \nu^{(2)}_a \text{ for } P_V\nu\text{-almost every } a \in V.
\end{equation}

Now, if $P_V$ is injective on a set of full $\nu$-measure, then by Lemma~\ref{lem:as injectivity Diracs}, we could conclude that $\nu_a$ is a Dirac's delta for $P_V \nu$-almost every $a \in V$. By \eqref{eq:nu sum slices}, this would imply that  $\nu^{(1)}_a = \nu^{(2)}_a$ for $P_V\nu$-almost every $a \in V$. This would make a contradiction with \eqref{eq:shifted slices}, as $t$ is non-zero and \eqref{eq:equal projections} gives $P_V \nu = P_V \nu^{(1)} = P_V \nu^{(2)}$.
\end{proof}
Proposition~\ref{prop:self-sim} is an immediate consequence of Proposition~\ref{prop:ifs with translate}, as any homogeneous iterated function system consisting of similarities which are not all equal must contain two distinct translation vectors $t_i$.

\section{Examples}\label{sec:examples}

In this section we provide examples showing that Theorems~\ref{thm:injectivity-proj} and~\ref{thm:inverses main} cannot be improved in the following directions:
\begin{itemize}
	\item[(a)] in assertion~\ref{it:hdim proj}, $\hdim X$ cannot be replaced by $\hdim \mu$,
	\item[(b)] in assertion~\ref{it:bdim proj}, $\udim X$ cannot be replaced by $\hdim X$ and the range of $\alpha$ cannot be extended (in terms of $\udim X$) to any larger interval in $(0, 1)$, possibly up to the endpoint,
	\item[(c)] in assertion~\ref{it:adim proj}, $\adim X $ cannot be replaced by $\udim X$.
\end{itemize}

Considering the setup of Theorem~\ref{thm:inverses main}, we construct suitable examples for a measure $\mu$ of compact support $X \subset \R^N$ and an open set of maps $\mathcal{U} \subset \Lip(X, \R^k)$, $k < N$, containing all full-rank linear maps $L \in \Lin(\R^N, \R^k)$ restricted to $X$. As every prevalent set is dense (see \cite[Fact~2']{Prevalence92}), the examples show that Theorem~\ref{thm:inverses main} is sharp (in the sense of the items (a)--(c)) also as a result on embeddings via prevalent maps in the space of Lipschitz or $C^r$, $r=1,2,\ldots, \infty$, maps into $\R^k$ (see Remark~\ref{rem:prevalence lipschitz}).

As described in Remark~\ref{rem:linear to projection}, every full-rank linear map $L \in \Lin(\R^N, \R^k)$, $k \le N$, can be uniquely represented as $L = \Psi\circ P_V$, where $V = (\Ker L)^\perp$ and $\Psi \colon V \to \R^k$ is a linear isomorphism, such that Lebesgue-almost all linear maps $L \in \Lin(\R^N, \R^k)$ correspond to $\gamma_{k,N}$-almost all orthogonal projections $P_V$, $V \in \Gr(k, N)$, for the rotation-invariant measure $\gamma_{k,N}$, and open sets in $\Lin(\R^N, \R^k)$ correspond to open sets in $\Gr(k, N)$. It follows that the examples are valid also within the setup of Theorem~\ref{thm:injectivity-proj}.

The following proposition provides examples corresponding to the item~(a).

\begin{prop}\label{prop:ex ball}
For every $1 \le k < N$ there exists a finite Borel measure in $\R^N$ with a compact support $X$ and an open set $\mathcal{U} \subset \Lip(X, \R^k)$, containing all full-rank linear maps $L \in \Lin(\R^N, \R^k)$ restricted to $X$, such that $\hdim \mu = 0$ and no map $\phi \in \mathcal{U}$ is injective on a full $\mu$-measure set with continuous inverse.
\end{prop}
\begin{proof}
Let $X = \overline{B_N(0,1)}\subset \R^N$ and let $Y$ be a countable dense subset of $X$. Consider a finite Borel measure $\mu$ on $Y$ such that $\mu(\{x\}) > 0$ for every $x \in Y$. Then the (topological) support of $\mu$ is equal to $X$ and $\hdim \mu = 0$. Let $L \in \Lin(\R^N, \R^k)$ be a full-rank linear map. As noted above, we can write $L = \Psi\circ P_V$, where $V = (\Ker L)^\perp$ and $\Psi \colon V \to \R^k$ is a linear isomorphism.

Choose $w \in V^\perp$ such that $\|w\| = \frac 1 2$ and let
\[
A_+ = (V +w) \cap X, \qquad A_- = (V -w) \cap X.
\]
Then $\dist(A_+, A_-) = 1$ and $B = P_V(A_+) = P_V(A_-)$ is a closed $k$-dimensional ball in $V$. Take $f \in \Lip(X, \R^k)$ such that $\|f\|_{\Lip} < \eps$ for a small $\eps > 0$, where the Lipschitz norm $\|\cdot\|_{\Lip}$ is defined in \eqref{eq:lip_norm}. Then $g = \Psi^{-1}\circ f \in \Lip(X, V) \simeq \Lip(X, \R^k)$ and $\|g\|_{\Lip}$ is arbitrarily small for sufficiently small $\eps$. Set $\psi = P_V + g\colon \R^N \to V$. As $P_V|_{A_\pm}$ is a translation onto $B$, the map $\psi|_{A_\pm}$ is bi-Lipschitz and 
\begin{equation}\label{eq:A_pm}
\psi(A_+) \cap \psi(A_-) \supset U \cap V
\end{equation}
for a non-empty open set $U \subset \R^N$, provided $\eps$ is chosen sufficiently small. Let $\widetilde A_+, \widetilde A_-$ be disjoint compact sets containing, respectively, some neighbourhoods of $A_+, A_-$. By \eqref{eq:A_pm} and the density of $Y$, there exists $x_+ \in \widetilde A_+ \cap Y$ such that $\psi(x_+) \in U \cap V$. Then, again by \eqref{eq:A_pm} and the density of $Y$, there exists $x_- \in \widetilde A_-$ such that $\psi(x_-) = \psi(x_+)$, and a sequence of points $x_n \in \widetilde A_- \cap Y$ with $x_n \to x_-$. 

Suppose $\psi$ is injective on a full $\mu$-measure subset of $X$. Then $\psi$ is injective on $Y$. However, we have $(\psi|_Y)^{-1}(\psi(x_+)) = x_+ \in \widetilde A_+$, $\psi(x_n) \to \psi(x_-) = \psi(x_+)$ and $(\psi|_Y)^{-1}(\psi(x_n)) = x_n \in \widetilde A_-$. As $\widetilde A_+, \widetilde A_-$ are disjoint compact sets, $(\psi|_Y)^{-1}$ cannot be continuous at $\psi(x_+)$.

Concluding, for every full-rank linear map $L \in \Lin(\R^N, \R^k)$ we have found $\eps = \eps(L) >0$ such that $\psi = P_V + g$ (where $L =\Psi\circ P_V$ and $g = \Psi^{-1}\circ f$) is not injective on a full $\mu$-measure set with continuous inverse for every $f \in \Lip(X, \R^k)$ with $\|f\|_{\Lip} < \eps$. Consequently, $\phi = L + f = \Psi\circ(P_V + g)$ is not injective on a full $\mu$-measure set with continuous inverse. This proves the proposition, as $\mathcal U$ can be taken as the union over $L$ of balls in $\Lip(X, \R^k)$ of radius $\eps(L)$, centred at $L$.

\end{proof}

Examples corresponding to the items (b)--(c) are inspired by \cite[Example~1]{FalconerHowroydProj} and \cite[Example~3.7]{SauerYorke97}, but differ in details and require different analysis than the ones presented in the cited works. The common construction in both cases (b)--(c) is as follows. Consider a pair of sequences $( r_i )_{i=1}^{\infty}, ( \ell_i )_{i=1}^{\infty}$ of positive real numbers, which decrease to $0$ with $\ell_i < r_i$ (in fact, $\ell_i$ will be chosen to decay much faster than $r_i$). Fix $N \in \N$, $1 \le k < N$ and let $e_1, \ldots, e_N$ be the standard orthonormal basis of $\R^N$. Denote 
\[
\Jk = \{ J \subset \{1, 2, \ldots, N\} : \# J = k + 1 \},
\]
and for $J \in \Jk$ let $S_J \subset \R^N$ be the $k$-dimensional unit sphere defined as 
\[
S_J = \big\{ (x_1, \ldots, x_N) \in \R^N : \sum \limits_{j \in J} x_j^2 = 1 \text{ and } \ x_j = 0 \text{ for } j \notin J\big\}.
\]
For each $J \in \Jk$ and $i \geq 1, $ let $X_{J,i} \subset \R^N$ be an $\ell_i$-separated set of maximal cardinality in 
\[
r_i S_J = \{ r_i x : x \in S_J \}.
\]
Then $X_{J,i}$ is an $\ell_i$-net  in $r_iS_J$ and 
\[
C_2\Big(\frac{r_i}{\ell_i}\Big)^k \le \# X_{J,i} \le C_1\Big(\frac{r_i}{\ell_i}\Big)^k
\]
for some constants $C_1, C_2 > 0$ independent of $i$ and $J$. Hence, $X_{J, i}$ can be seen as (approximately) uniformly distributed in $r_iS_J$. For $J \in \Jk$ we define a compact set $X_J \subset \R^N$ by
\[ X_J = \bigcup \limits_{i = 1}^\infty X_{J, i} \cup \{ 0 \}  \]
and set
\[ X = \bigcup \limits_{J \in \Jk} X_J.\]

To show the items (b)--(c), we prove the following proposition.

\begin{prop}\label{prop:ex spheres} For every $1 \le k < N$ and the set $X$ defined above, there exists an open set $\mathcal{U} \subset \Lip(X, \R^k)$ containing all full-rank linear maps $L \in \Lin(\R^N, \R^k)$ restricted to $X$, such that the following hold.
\begin{enumerate}[$($i$)$]
\item For $r_i = 2^{-i}$ and $\ell_i = 2^{-i^2}$, we have $\hdim X = 0$ and no map $\phi \in \mathcal{U}$ is injective on $X$ with a pointwise $\alpha$-H\"older inverse for any $\alpha > 0$.
\item\label{it:box holder bound} For every $s \in (0, k)$ and $t = \frac{k}{k-s}$, $r_i = 2^{-i}$, $\ell_i = 2^{-it}$, we have $\ldim X = \udim X = s$ and no map $\phi \in \mathcal{U}$ is injective on $X$ with a pointwise $\alpha$-H\"older inverse for $\alpha >  1 - \frac{\udim X}{k}$.
\end{enumerate}
In particular, if $\mu$ is a finite Borel measure supported on $X$ with $\mu(\{ x\}) > 0$ for every $x \in X$, in both cases we conclude that there is no full $\mu$-measure set on which $\phi$ is injective with a pointwise $\alpha$-H\"older inverse for $\alpha$ in the given range.
\end{prop}

Note that as Proposition~\ref{prop:ex spheres}\ref{it:box holder bound} provides a compact set $X$ with $\udim X < k$ and an open set of Lipschitz maps (containing all full-rank linear maps) without a pointwise $\alpha$-H\"older inverse on $X$ for some $\alpha \in (0,1)$, we conclude that in the assertion~(iii) of Theorems~\ref{thm:injectivity-proj} and~\ref{thm:inverses main}, the Assouad dimension cannot be replaced by the upper box-counting dimension.

The proof of Proposition~\ref{prop:ex spheres} is preceded by two lemmas.
\begin{lem}\label{lem:sphere dim}
For any choice of the sequences $r_i \searrow 0$ and $\ell_i < r_i$, the set $X$ satisfies $\hdim X = 0$. For $r_i = 2^{-i}$ and $\ell_i = 2^{-ti}$, $t>1$, we have $\ldim X = \udim X = k\frac{t-1}{t}$.
\end{lem}

\begin{proof}
The first statement is immediate as $X$ is countable. For the second one,  take $J \in \Jk$, $r_i = 2^{-i}$, $\ell_i = 2^{-ti}$, $t>1$, and note that for a small $r>0$ one has the following bounds on the covering number $N(X_{J,i}, r)$, coming, respectively, from covering each point of $X_{J, i}$ separately, and covering the whole sphere $r_i S_J$:
\begin{align}
\label{eq: first cover}
N(X_{J,i}, r) &\leq \# X_{J,i} \leq C_1 \Big(\frac{r_i}{\ell_i}\Big)^k = C_1 2^{k(t-1)i},\\
\label{eq: second cover} N(X_{J,i}, r) &\leq C\Big(\frac{r_i}{r}\Big)^k = \frac{C}{r^k 2^{ki}},
\end{align}
with a constant $C>0$ independent of $i$ and $J$.

Let $m = m(r)$ be such that $2^{-m} \leq r < 2^{-m+1}$. Note that $X_{J,i} \subset B(0, r)$ for $i > m$. Hence,
applying the bounds~\eqref{eq: first cover} for $i \le \lceil m / t \rceil$ and~\eqref{eq: second cover} for $i > \lceil m / t \rceil$ we obtain
\[
N( X_J, r) \leq 1 + C_1\sum_{i=1}^{\lceil m / t \rceil}  2^{k(t-1)i} + C\sum_{i=\lceil m / t \rceil +1}^m  2^{k(m-i)} \leq C'_1 2^{km (t-1) / t}
\]
for some constant $C'_1 > 0$ and every sufficiently small $r > 0$.
This gives
\[ \udim X_J \leq \lim_{m\to\infty} \frac{\log C'_1 + (km(t-1)/t) \log 2}{(m -1)\log 2} = k\frac{t-1}{t}. \]
On the other hand, for $i = \lfloor (m-2) / t \rfloor$ we have $2r \le \ell_i$, so the set $X_{J,\lfloor (m-2) / t \rfloor}$ is $2r$-separated, and hence
\begin{align*}
N( X_J, r) &\ge N(X_{J,\lfloor (m-2) / t \rfloor}, r) \ge \#X_{J,\lfloor (m-2) / t \rfloor} \ge  C_2\bigg(\frac{r_{\lfloor (m-2) / t \rfloor}}{\ell_{\lfloor (m-2) / t \rfloor}}\bigg)^k\\
&= C_2 2^{k(t-1)\lfloor (m-2) / t \rfloor} \ge C'_2 2^{km (t-1) / t}
\end{align*}
for some constant $C'_2 > 0$ and sufficiently large $m$. Thus,
\[
\ldim X_J \geq \lim_{m\to\infty} \frac{\log C'_2 + (km(t-1)/t) \log 2}{(m -1)\log 2} = k\frac{t-1}{t}.
\]
As $X$ is a finite union of $X_J, J \in \Jk$, we have $\ldim X = \udim X = k\frac{t-1}{t}$.
\end{proof}

\begin{lem}\label{lem:petrurbed image}
For every full-rank linear map $L \in \Lin(\R^N, \R^k)$, there exist $J=J(L) \in \Jk$ and $\eps = \eps(L) > 0$ such that for every $f \in \Lip(X, \R^k)$ with $\|f \|_{\Lip} < \eps$ and every $0 < r \leq 1$ there exists $y \in rS_J$ with $(L + f)(y) = (L + f)(0)$.
\end{lem}

\begin{proof}
Fix a full-rank linear map $L \in \Lin(\R^N, \R^k)$. As previously, write $L = \Psi\circ P_V$, where $V = (\Ker L)^\perp$ and $\Psi \colon V \to \R^k$ is a linear isomorphism. Since $\{e_1, e_2, \ldots, e_N \}$ is a linear basis of $\R^N$ and $\dim V^\perp = N - k$, there exists $J \in \Jk$ such that $\Span(V^\perp, \{e_j : j \in J\}) = \R^N$ (actually, we can obtain this with a subset of $J$ of cardinality $k$, but it is crucial for the construction that $\# J = k+1 $). Note that $\Span(S_J) = \Span(\{e_j : j \in J\})$ and take any $y_0 \in V^\perp \cap S_J$. Such $y_0$ exists, as otherwise $V^\perp \cap \Span(S_J) = \{ 0 \}$, but this is impossible since $\dim V^\perp= N - k$ and $\dim(\Span(S_J)) = k + 1$. Let now $W$ be the orthogonal complement of $y_0$ in $\Span(S_J)$, so that $W \subset \Span(S_J)$ and $\dim W = k$. Note that $W \cap V^\perp = \{ 0\}$, as $\dim W + \dim V^\perp = N$ and $\Span(W, V^\perp) = \Span(W, y_0, V^\perp) = \Span(S_J, V^\perp) = \R^N$. As $y_0 \in V^\perp = \Ker P_V$, we have $P_V (y_0) = 0$.  Since $W \cap \Ker P_V = W \cap V^\perp = \{ 0 \}$, we see that $P_V$ is injective on $W$. As $\dim W = k =\dim V$, we conclude that $P_V$ is a linear isomorphism between $W$ and $V$ and hence it is a bi-Lipschitz homeomorphism.

Now we show that
\begin{equation}\label{eq:claim}
\text{$P_V$ is a bi-Lipschitz homeomorphism from $S_J \cap B(y_0, \delta)$ onto its image for some $\delta > 0$.}
\end{equation}
To prove \eqref{eq:claim}, we first note that $W$ is the tangent space to $S_J$ at $y_0$, as $W$ is the orthogonal complement of $y_0$ in $\Span(S_J)$. Therefore, for any sufficiently small $\eta > 0$ there exists $\delta>0$ such that the map $S_J \cap B(y_0, \delta) \ni x \mapsto P_W(x - y_0) \in W$ is a  bi-Lipschitz homeomorphism onto its image and such that $h \colon  S_J \cap B(y_0, \delta)  \to \R^N$ given by $h(x) = x - y_0 - P_W(x - y_0)$ satisfies $\|h\|_{\Lip} < \eta$. As $x = y_0 + P_W(x-y_0) + h(x)$, we have 
\begin{equation}\label{eq:proj decomp} P_V (x) = P_V(P_W(x - y_0)) + P_V (h(x)).
\end{equation}
for $x \in S_J \cap B(y_0,\delta)$. 
As $x \mapsto P_W(x - y_0)$ is bi-Lipschitz as a map from $S_J \cap B(y_0, \delta)$ to $W$ and $P_V$ is bi-Lipschitz on $W$, we have that $x \mapsto P_V(P_W(x - y_0))$ is bi-Lipschitz on $S_J \cap B(y_0, \delta)$. On the other hand, $\|P_V \circ h\|_{\Lip} < \eta$, so we see from \eqref{eq:proj decomp} that $P_V$ is bi-Lipschitz on $S_J \cap B(y_0, \delta)$ if $\eta$ is small enough (compared to the Lipschitz constants of $x \mapsto P_V(P_W(x - y_0))$ and its inverse on $S_J \cap B(y_0, \delta)$, which are uniformly bounded away from $0$ and $\infty$ for all small $\delta$). Such $\eta$ exists for $\delta$ small enough, which shows \eqref{eq:claim}.

Take $y_0 \in S_J$ and $\delta > 0$ as in \eqref{eq:claim}. Then, by a homothetic change of coordinates, $P_V$ is a bi-Lipschitz homeomorphism from $rS_J \cap B(ry_0, r\delta)$ onto its image for every $0 < r \leq 1$, with Lipschitz constants independent of $r$. As previously, note that if $f \in \Lip(X, \R^k)$ satisfies $\|f\|_{\Lip} < \eps$ for a small $\eps > 0$, then $g = \Psi^{-1}\circ f \in \Lip(X, V) \simeq \Lip(X, \R^k)$ and $\|g\|_{\Lip}$ is arbitrarily small for sufficiently small $\eps$. 
Consequently, setting $\psi = P_V + g$, we see that there exists $\eps = \eps(L) >0$ such that if $\|f\|_{\Lip} < \eps$, then $\psi$ is a bi-Lipschitz homeomorphism from $rS_J \cap B(ry_0, r\delta)$ onto its image with uniform Lipschitz constants, i.e. there exist $c,C > 0$ such that
\begin{equation}\label{eq:bi-Lip}  c \|x-y\| \leq  \| \psi(x) - \psi(y) \| \leq C \|x-y\|
\end{equation}
for every $0 < r \leq 1$ and $x,y \in rS_J \cap B(ry_0, r\delta)$.
Our goal is to show that $\psi(0) = g(0)$ belongs to the image $\psi(rS_J \cap B(ry_0, r\delta))$. To see this, note that as $\psi$ is a bi-Lipschitz homeomorphism of the $k$-dimensional ball $rS_J \cap B(ry_0, r\delta)$,  \eqref{eq:bi-Lip} implies that $\psi(rS_J \cap B(ry_0, r\delta))$ contains a ball in $\R^k$ of radius $cr\delta$ centred at $\psi(r y_0) = g(ry_0)$. As $\|g(ry_0) - g(0)\| \leq \eps r$, we see that $g(0) \in \psi(rS_J \cap B(ry_0, r\delta))$ provided that $\eps$ is small enough to satisfy $\eps < c\delta$.
Hence, $\psi(y) = \psi(0)$ for some $y \in rS_J$. As $\psi = P_V + g$, $L = \Psi\circ P_V$ and $f = \Psi\circ g$, this implies $(L + f)(y) = (L + f)(0)$, which ends the proof of the lemma.
\end{proof}

\begin{proof}[Proof of Proposition~\rm\ref{prop:ex spheres}]
As previously, we set $\mathcal{U}$ to be the union over all full-rank linear maps $L \in \Lin(\R^N, \R^k)$ of balls in $\Lip(X, \R^k)$ with radius $\eps(L)$, centred at $L$, where $\eps(L)$ is as in Lemma~\ref{lem:petrurbed image}. Consider $\phi = L + f \in \mathcal{U}$, where $L \in \Lin(\R^N, R^k)$ and $\|f\|_{\Lip} < \eps(L)$. By Lemma~\ref{lem:petrurbed image}, for $i \geq 1$ and $J=J(L)$, we can find $y_i \in r_i S_J$ with $\psi(y_i) = \psi(0)$. Let $x_i \in X_{J, i}$ be such that $\|x_i - y_i \| \leq \ell_i$ (recall that $X_{J, i}$ is an $\ell_i$-net in $r_iS_J$). Suppose that $\phi$ is injective on $X$ with a pointwise $\alpha$-H\"older inverse. Then there exists $M > 0$ such that for every $i \geq 1$,
\[
\begin{split}r_i &= \|x _i \| \leq M \|\phi(x_i) - \phi(0)\|^{\alpha} = M \|\phi(x_i) - \phi(y_i)\|^{\alpha} \\
&\leq M (\|L(x_i) - L(y_i)\| + \|f(x_i) - f(y_i)\|)^{\alpha}\\
&\leq M(\|L\|+\eps)^\alpha \|x_i - y_i\|^\alpha \leq M(\|L\|+\eps)^\alpha \ell_i^\alpha,
\end{split} \]
so
\[ \frac{r_i}{\ell_i^{\alpha}} \leq M (\|L\|+\eps)^\alpha. \]
This however cannot be the case if $\lim_{i \to \infty} \frac{r_i}{\ell_i^{\alpha}} = \infty$. Therefore, for $r_i = 2^{-i}$, $\ell_i = 2^{-i^2}$ we conclude that $\phi$ is not injective on $X$ with a pointwise $\alpha$-H\"older inverse for any $\alpha > 0$. This proves the first assertion of the proposition. For the choice $r_i = 2^{-i}$, $\ell_i = 2^{-ti}$, $t > 1$, we obtain that $\phi$ is not injective on $X$ with a pointwise $\alpha$-H\"older inverse for any $\alpha > 1/t$. Setting $t = \frac{k}{k-s}$ for $s \in (0,k)$, we obtain $\frac 1 t = 1 - \frac{s}{k}$ and $\ldim X = \udim X = k\frac{t-1}{t} = s$ by Lemma~\ref{lem:sphere dim}. This finishes the proof.
\end{proof}

\section*{Acknowledgements}
Some parts of Theorem~\ref{thm:inverses main} are inspired by an unpublished note \cite{KSnote} by Istv\'an Kolossv\'ary and K\'aroly Simon. We thank them for sharing it with us. We are grateful to Pertti Mattila for reading an early version of this manuscript and valuable remarks. We thank Bal\'azs B\'ar\'any and Boris Solomyak for useful discussions. KB and A\'S were partially supported by the National Science Centre (Poland) grant 2019/33/N/ST1/01882. YG was partially supported by the National Science Centre (Poland) grant 2020/39/B/ST1/02329.

\bibliographystyle{alpha}
\bibliography{universal_bib}

\end{document}